\documentclass[onefignum,onetabnum]{siamonline220329}

\usepackage{amsfonts}
\usepackage{amsmath}
\usepackage{amssymb}
\usepackage{graphicx}
\usepackage{float}
\usepackage{subcaption}
\usepackage{epstopdf}
\usepackage{algorithmic}
\ifpdf
  \DeclareGraphicsExtensions{.eps,.pdf,.png,.jpg}
\else
  \DeclareGraphicsExtensions{.eps}
\fi

\DeclareMathOperator{\tr}{tr}

\usepackage{enumitem}
\setlist[enumerate]{leftmargin=.5in}
\setlist[itemize]{leftmargin=.5in}

\newsiamremark{remark}{Remark}
\newsiamremark{hypothesis}{Hypothesis}
\crefname{hypothesis}{Hypothesis}{Hypotheses}
\newsiamthm{claim}{Claim}

\headers{Deep Linear Networks for Matrix Completion -- An Infinite Depth Limit}{Nadav Cohen, Govind Menon, and Zsolt Veraszto}

\title{Deep Linear Networks for Matrix Completion -- An Infinite Depth Limit\thanks{Submitted to the editors DATE.}}

\author{Nadav Cohen\thanks{Blavatnik School of Computer Science, Tel Aviv University, Israel 
  (\email{cohennadav@tauex.tau.ac.il}).}
\and Govind Menon\thanks{Division of Applied Mathematics, Brown University, Providence, RI 
  (\email{govind\_menon@brown.edu}, \email{zsolt\_veraszto@brown.edu}).}
\and Zsolt Veraszto\footnotemark[3]}

\usepackage{amsopn}

\ifpdf
\hypersetup{
  pdftitle={An Example Article},
  pdfauthor={Nadav Cohen, Govind Menon, and Zsolt Veraszto}
}
\fi

\begin{document}

\maketitle

\begin{abstract}
The deep linear network (DLN) is a model for implicit regularization in gradient based optimization of overparametrized learning architectures. Training the DLN corresponds to a Riemannian gradient flow, where the Riemannian metric is defined by the architecture of the network and the loss function is defined by the learning task.  We extend this geometric framework, obtaining  explicit expressions for the volume form, including the case when the network has infinite depth. We investigate the link between the Riemannian geometry and the training asymptotics for matrix completion with rigorous analysis and numerics. We propose that under small initialization, implicit regularization is a result of bias towards high state space volume.
\end{abstract}

\begin{keywords}
generalizability, implicit regularization, Riemannian gradient flow, Deep Linear Network, matrix completion.
\end{keywords}

\begin{MSCcodes}
68T07, 
58D17, 
37N40 
\end{MSCcodes}

\section{Introduction}
\subsection{The deep linear network}
Deep learning has proven its general applicability in several fields of applied science in the past decade \cite{LeCun2015}. While neural networks are structurally simple function approximators (e.g.\ \cite{Lapedes1988}), several questions around them remain unanswered. Two fundamental problems are to obtain first principles explanations for generalizability and implicit regularization. Generalizability refers to a network's performance on new, previously unseen data. Implicit regularization is the feature of deep networks to avoid overfitting despite overparametrization. In the context of classical regression problems, overfitting is mitigated by explicit regularization. Deep learning architectures are observed not to overfit despite the lack of explicit regularization. This is referred to as implicit regularization.



A simple model in which the effect of overparametrization in deep networks can be studied is the deep linear network (DLN)~\cite{Arora2018,arora2019implicit}. The DLN is simple enough to serve as a minimal model for neural networks. It may also be applied directly to optimization problems, such as matrix completion\cite{candes2008exact}, autoencoders \cite{Bah2019} and multi-task training. The most common approach to the training process is gradient based minimization of a loss function, a process known as empirical risk minimization. 

Recent work has shown that the DLN is also of intrinsic mathematical interest. Training the DLN corresponds to a gradient flow of a loss function with a Riemannian structure determined by the architecture of the network. The purpose of this paper is to further develop the mathematical theory of DLN. The main new contributions are explicit formulas for volume forms, including the infinite depth limit, that parallel corresponding expressions in random matrix theory. A numerical investigation that interprets some aspects of these formulas is also presented. We define the DLN below and then state these results more precisely.


\subsection{The dynamical systems}
Let $\mathbb{M}_d$ denote the space of real $d\times d$ matrices. The state space of the
DLN consists of $N$ weight matrices, $W_i\in \mathbb{M}_d$, $1 \leq i \leq N$, which we denote by
\begin{equation}
\label{eq:def-bigW}
    \mathbf{W}=\left(W_N, W_{N-1},\dots, W_1\right).
\end{equation}
The number of weight matrices, $N$, is referred to as the {\em depth\/} of the network. We have assumed that all matrices have the same size for simplicity. In fact, all that is required is that the matrix sizes be such that the product $W$ in equation~\eqref{eq:endend} is well-defined. Properties of the DLN in such generality have been studied in~\cite{arora2019implicit,Bah2019}. We build on these results in our work, but restrict ourselves to $W_i\in \mathbb{M}_d$ for ease of computation. Some aspects of our analysis extend to a manifold of fixed rank matrices as in~\cite{Bah2019} (see Section \ref{sec:rankoneexample}).

The training process is an optimization problem for the weight matrices based on a lower dimensional observable. This observable, referred to as the {\em end-to-end matrix\/}, is the product of the weight matrices
\begin{equation}
\label{eq:endend}
    W=\pi\left(\mathbf{W}\right):=\prod_{i=N}^1 W_i.
\end{equation}
Notice that the dimension of $W$ is $d^2$ regardless of the network depth. 

The gradient flow for a {\em loss function\/} $E(W)$
\begin{equation}
\label{eq:gradflow1}
    \frac{d}{dt} W_j=-\partial_{W_j} E\left( W\right), \quad 1 \leq j \leq N,
\end{equation}
is used as our training model. 
Equation~\eqref{eq:gradflow1} is an idealization for the manner in which training is implemented in practice. Since $W$ is defined through equation~\eqref{eq:endend}, we also find that
\begin{equation}
    \dot W_j=-W_{j+1}^T \dots W_{N}^T \,\partial_W E(W) \,W_{1}^T \dots W_{j-1}^T, \quad j=1,\dots, N.
\end{equation}

Here and below, the notation $\partial_A f(A)$ for a differentiable function $f: \mathbb{M}_d \rightarrow \mathbb{R}$ denotes the Euclidean gradient, that is, for every $B \in \mathbb{M}_{d}$, $df(A)(B)=\tr \left( \partial_A f(A)^T B\right)$. 

Our interest lies in the long-time behavior of the observable $W(t)$. The limit $\lim_{t\to \infty} W(t)$ is called the {\em training outcome\/}. Since the dynamics is governed by a gradient flow, this limit exists when the loss function $E(W)$ is bounded below and has compact sublevel sets. However, natural loss functions, such as those for certain matrix completion problems, do {\em not\/} have compact sublevel sets. In fact, the minima of these loss functions are submanifolds of $\mathbb{M}_d$. Thus, the prediction of training outcomes is a subtle problem. The existence of $\lim_{t\to \infty} W(t)$ when each of the $W_i$ has full rank was established using Lojasiewicz's theorem~\cite[Thm 10]{Bah2019}. In order to use this convergence theorem as a foundation for the analysis in this paper, we restrict ourselves to $W_i \in GL(d)$ for most of the analysis. The space $GL(d)$ is the set of bijective linear transformations of $\mathbb{R}^d$, which is isomorphic to the set of $d$-by-$d$ invertible matrices. Despite this assumption, training outcomes for many matrix completion problems have low rank. The repeated appearance of low rank matrices as training outcomes has stimulated the use of several heuristics for the prediction of training outcomes~(see the discussion in~\cite{maynotbenorms}). The goal of this paper is to shed light on this question using the Riemannian geometry of the DLN and numerical simulations.


Let us now review the Riemannian geometry underlying the Euclidean gradient flow defined in  equation~\eqref{eq:gradflow1}. An important concept, see~\cite[Defn.1]{Arora2018}, is the notion of 
{\em balancedness\/}. We say that the weight matrices $W_j$ and $W_{j+1}$ are $G_j$-balanced if
\begin{equation}
    G_j:=W_{j+1}^T W_{j+1}-W_j W_j^T.
\label{gbalancedness}
\end{equation}
For fixed $G_j$, $1\leq j \leq N$, equation~\eqref{gbalancedness} defines an algebraic variety $\mathbb{M}_d^N$ that is stratified by the rank. When $G_j=0$, for each value of the rank $r$, the gradient flow~\eqref{eq:gradflow1} leaves the manifold $\mathcal{M}_r$ of rank-$r$ matrices solving~\eqref{gbalancedness} invariant~\cite[Cor.6]{Bah2019}. These are called the {\em balanced manifolds\/}. It is immediate from equation~\eqref{gbalancedness} that the singular values of the weight matrices $\{W_j\}_{j=1}^N$ are equal on the balanced manifolds. 

In order to use existing convergence theory, we restrict attention to full-rank matrices, and we denote $\mathcal{M}_d$ by $\mathcal{M}$ in what follows.
This balanced manifold allows to separate the dynamics into a flow `upstairs' in $\mathbb{M}_d^{N}$, described by equation~\eqref{eq:gradflow1}, and a flow `downstairs' for the end-to-end matrix $W(t)$ in $GL(d)$. The flow downstairs is a Riemannian gradient flow with a metric computed in~\cite{Bah2019} that may be described as follows.
We define the linear map $\mathcal{A}_{N,W}: T_W GL(d) \simeq \mathbb{M}_d \rightarrow \mathbb{M}_{d}$ 
\begin{equation}
\label{eq:pushforward1}    
\mathcal{A}_{N,W}(Z) := \frac{1}{N} \sum_{j=1}^N \left(W W^T\right)^\frac{N-j}{N} Z \left(W^T W\right)^\frac{j-1}{N}.
\end{equation}
On the balanced manifold the end-to-end matrix satisfies the Riemannian gradient flow
\begin{equation}
\label{eq:gradient-flow}
    \dot W= -\mathrm{grad}_{g^N} E(W),
\end{equation}
under the metric
\begin{equation}
\label{metric}
    g^N(Z_1,Z_2)=\tr\left(\mathcal{A}_{N,W}^{-1}\left(Z_1\right)^T Z_2\right),
\end{equation}
where $Z_1,Z_2 \in T_W GL(d)$. This structure allows us to extend the DLN geometry to the infinite depth limit, with the linear operator
\begin{equation}
    \mathcal{A}_{\infty,W}\left(Z\right)=\lim_{N\rightarrow \infty} \mathcal{A}_{N,W}(Z) = \int_0^1\left(W W^T\right)^{\left(1-\tau\right)} Z \left(W^T W\right)^\tau d\tau,
    \label{A_N}
\end{equation}
replacing $\mathcal{A}_N$ in equation~\eqref{metric} to define a limiting metric $g^\infty$. 
When $N$ is finite, the dynamical system~\eqref{eq:gradient-flow} corresponds to a flow upstairs in $\mathbb{M}_d^N$. In the limit of infinite depth, equation~\eqref{eq:gradient-flow} continues to hold, even though there is no longer a well-defined flow upstairs.

The existence of the infinite depth metric, in particular its resemblance to the Bogoliubov inner product in quantum statistical mechanics, was noted in~\cite[Remark 8]{Bah2019}. We develop some properties of this metric below, emphasizing explicit formulas in SVD coordinates in Section~\ref{subsec:metric-volume}. These formulas also show that the metric is at least $C^1$ in the open subset of $\mathbb{M}_d$ of matrices with distinct singular values. These formulas are better seen as first steps towards deeper exploration. We still lack an understanding of the curvature and geodesics of these metric though some partial results have been obtained in~\cite{Zsolt-thesis}. The appearance of such elegant Riemannian structures in the DLN are surprising at first sight. However, it is helpful to note that fundamental interior point methods for conic programs also have a surprising gradient structure (see~\cite{Bayer1989I,Bayer1989II,hildebrand}). In recent work~\cite{MY-conic}, one of the authors and Yu have studied these metrics, seeking precise comparisons between gradient flows underlying conic programs and deep learning.

The above structure applies to an arbitrary loss function. In practice, the loss function is determined by the learning task and the ease of computation. We focus on matrix completion in this paper, choosing the loss function $E(W)$ to be the quadratic distance from a fixed matrix $\Phi$ termed the {\em optimization objective\/} \cite{Gunasekar2017}. Using $\circ$ for element-wise (Hadamard) product, the family of loss functions we study takes the form
\begin{equation}
    E_\mathcal{B}(W)=\frac{1}{2}\|\mathcal{B}\circ\left(\Phi-W\right)\|_2^2.
    \label{loss}
\end{equation}

Here $\mathcal{B}$ is a matrix whose entries are either zero or one. This notation allows us to include several forms of matrix completion. An important example is the following: Let $\mathcal{B}=I$ select the diagonal elements, and consider
\begin{equation}
    E_I(W)=\frac{1}{2}\|\text{diag}\left(\Phi-W\right)\|_2^2
    \label{eq:diagonalenergy}
\end{equation}
Here $\text{diag}(\cdot)$ returns a diagonal matrix constructed from the diagonal elements of its arguments. 

The nature of the loss function is determined by the matrix $\mathcal{B}$. For example, the energy defined by~\eqref{eq:diagonalenergy} has a submanifold of global minima. Indeed, given a diagonal matrix $\Phi$, $E_I(W)$ vanishes when $W$ is chosen to be any matrix whose diagonal entries are $\Phi$. Moreover, since the matrix can be completed to any rank from $1$ to $d$, some of these minima are noninvertible matrices. We consider numerical examples with several such $\mathcal{B}$.

\subsection{Statement of Results}
\subsubsection{The metric and volume forms}
\label{subsec:metric-volume}
The singular value decomposition (SVD) of $W$ is denoted $W=U\Sigma V^T$, where $U,V \in O(d)$ and $\Sigma$ denotes the diagonal matrix of singular values. We write $\Sigma_{ii}=\sigma_i$ and order the singular values in decreasing order: $\sigma_i\geq\sigma_j$ if $i<j$. 
The metric $g^N$ may be expressed in a simple manner using the SVD. We find (see equation~\eqref{metricmatrix} below) that 
\begin{equation}
\label{eq:metricmatrix-a}
g^N = (V\otimes U) D^N(\Sigma) (V\otimes U)^T. 
\end{equation}
Here $V\otimes U$ is the Kronecker product of $V$ and $U$ and $D_N \in \mathbb{R}^{d^2 \times d^2} $ is a diagonal operator with nonzero entries 
\begin{equation}
\label{eq:metricmatrix-c}
D^N_{il} = \frac{N}{\sum_{j=1}^N (\sigma_i^2)^{N-j/N} (\sigma_l^2)^{j/N}}, \quad 1\leq i,l \leq d.
\end{equation}
Here the subscript $il$ denotes a double index on the diagonal elements of $D^N$. The expression is unambiguous due to the symmetry of \eqref{eq:metricmatrix-c} in $i$ and $l$. The expression~\eqref{eq:metricmatrix-a} holds in the limit $N=\infty$, with
\begin{equation}
\label{eq:metricmatrix-b}
D^\infty_{il} =  \frac{2 \log (\sigma_i/\sigma_l)}{\sigma_i^2 -\sigma_l^2} \quad i\neq l, \quad D_{ii} = \frac{1}{\sigma_i^2}.
\end{equation}

These expressions for the metric are used to compute the associated volume forms in $GL(d)$. 
Let $\mathrm{van}(\Lambda)$ denote the Vandermonde determinant of a diagonal matrix $\Lambda$, let $d\Sigma$ denote Lebesgue measure on $\mathbb{R}^N$ and let $dU$ and $dV$ denote Haar measure on $O(d)$.
\begin{theorem}
\label{thm:volumeforms}
    The volume form of $g^N$ is given by
    \begin{align}
      \sqrt{\det g^N}dW=N^\frac{d(d-1)}{2}\,\det(\Sigma^2)^{\frac{1-N}{2N}}\,\mathrm{van}\left(\Sigma^{2/N}\right) \,d\Sigma dU dV.
    \end{align}
In the limit $N\to \infty$, the volume form of $g^\infty$ is
    \begin{align}
      \sqrt{\det g^\infty}dW=\frac{\mathrm{van}\left(\log \Sigma^2\right)}{\sqrt{\det\left(\Sigma^2\right)} } \,d\Sigma dU dV.
      \label{eq:volumeforminf}
    \end{align}
 \end{theorem}
The volume forms allow us to quantify the importance of regions of high volume that correspond to empirical observations of training outcomes. Notice that the volume density blows up when passing to the limit $\sigma_i \rightarrow 0$, showing a clear relationship between low rank and high volume. 

The reader unfamiliar with Riemannian geometry should note that all our work reduces to explicit calculations in SVD coordinates. The symmetries of the metric implicit in equation~\eqref{eq:metricmatrix-b} allow us to calculate several Jacobian determinants explicitly. The main subtlety in using SVD coordinates is that naive calculations must be restricted to the open set of $\mathbb{M}_d$ where $W$ has distinct singular values, and the branches of the SVD must be resolved on the lower-dimensional varieties corresponding to repeated eigenvalues. Formulas, such as those in Theorem~\ref{thm:volumeforms}, are established under the assumption that the singular values are distinct, and then seen to hold in the limit of repeated singular values using continuity. 

Such calculations follow the spirit of random matrix theory~\cite{Govind_RMT}. For example, Theorem~\ref{thm:volumeforms} suggests interesting asymptotics for the DLN in the limits $N\to\infty$ and $d\to \infty$. We do not consider this question in this paper, but see~\cite{hanin} for a similar investigation.

\subsubsection{Normal hyperbolicity}

The spectral decomposition of $g^\infty$ given in equation~\eqref{eq:metricmatrix-a} (with $N=\infty)$ and \eqref{eq:metricmatrix-b} has important consequences for the dynamics given by equation~\eqref{eq:gradient-flow}. For different choices of $\mathcal{B}$, let $\mathcal{N}_\mathcal{B}$ denote the set of global minima of the energy function~\eqref{loss},
\begin{equation}
    \mathcal{N}_\mathcal{B}=\{W: W\in \mathbb{M}_d,\, \mathcal{B}\circ \left(\Phi-W\right)=0\}.
\end{equation}
Recall that $\mathcal{B}$ is a matrix whose entries are either zero or one. Let $K$ denote the number of zeros in $\mathcal{B}$, 
\begin{equation}
    K=d^2-\sum_{i,j=1}^d \mathcal{B}_{ij}.    
\end{equation}
The set $\mathcal{N}_\mathcal{B}$ is an open submanifold of $GL(d)$ with dimension $K$. 
\begin{theorem}
\label{thm:nhim}
Any $K$-dimensional compact submanifold of $\mathcal{N}_\mathcal{B}$ is normally hyperbolic under the dynamical system~\eqref{eq:gradient-flow} for $N=\infty$.
\end{theorem}

The statement is a consequence of Lemma~\ref{lem:rates}. The proof is presented in  Section~\ref{sec:attraction}.

\subsubsection{Are the training outcomes low rank matrices?}
The origin of implicit regularization was proposed to arise from (quasi-)norms in~\cite{regularizationbyrank}. This idea is motivated by classical regression, where overfitting effects are often mitigated by adding explicit regularization terms to the optimization problem. Previous work in this direction tried to explicitly construct the regularizers in the form of  (quasi-)norms. This explanation was challenged in~\cite{maynotbenorms}, where the training outcomes were observed to be low rank matrices, even though matrix norms blew up. We develop this idea in Section \ref{sec:matrixcompletion}. On one hand, we view the bias towards low rank matrices as an entropic effect determined by the volume forms above (the volume forms diverge as the singular values approach zero). On the other hand, we construct numerical examples where using low rank as a criterion does not accurately predict the training outcome under small initial conditions. The numerical examples show that within the set of low rank minimizers, the actually observed ones are also the ones with the largest concentration of volume around them.




\subsubsection{Numerical simulations for finite and infinite depth}
Section \ref{sec:matrixcompletion} details numerical simulations of system \eqref{eq:gradient-flow} for $N\leq \infty$. As equation~\eqref{eq:pushforward1} is very costly computationally for large $N$ and \eqref{A_N} can only be explicitly evaluated in terms of SVD coordinates, \eqref{eq:gradient-flow} is inefficient to simulate directly. To keep track of the evolution of SVD coordinates, we derive their dynamics directly under the flow of \eqref{eq:gradient-flow}.

In the first few examples, we make use of energy function \eqref{eq:diagonalenergy}. In Section \ref{sec:wideexample}, we demonstrate that for large $N$ and diagonal matrix completion, bias towards low rank should be viewed as bias towards {\em minimal\/} rank. This example also motivates the study of smaller matrices, where the matrix size and the minimal rank of minimizers are comparable. The example in Section \ref{hyperbolas} is on $2\times 2$ matrices; yet it illustrates the effect of depth, demonstrating the validity of the infinite depth limit.

The examples of Section \ref{sec:finiterankdefminimizers} show simulation outcomes for different choices of energy functions in the family \eqref{loss}. For these examples we chose $\mathcal{B}$ such that $E_\mathcal{B}$ has a finite number of rank deficient minimizers. We see that not all of the rank deficient minimizers are actually observed as training outcomes, contradicting the idea that low rank accurately predicts simulation outcomes. Instead, we find that the observed minimizers are the minimum rank minimizers with maximal volume.

Section \ref{sec:rankoneexample} details an example where the entire state space has minimal rank. In this case, it is impossible to to predict simulation outcomes in terms of bias towards low rank. We demonstrate that high state space volume remains predictive, and the simulation outcomes are clustered in the region of state space where the volume is maximal. 

\subsection{Organization of the paper}
The rest of this paper is organized as follows. In Section 2, we review the Riemannian geometry of the DLN and establish equations~\eqref{eq:metricmatrix-a}--\eqref{eq:metricmatrix-c}. We then extend this geometry to the infinite depth limit and prove Theorem~\ref{thm:volumeforms}. In Section 3 we use numerical experiments to show the correspondence between state space volume and implicit regularization. We also present several comparisons between finite depth networks and the infinite depth limit. Our conclusions are summarized in Section~\ref{sec:conclusion}.

\section{The Riemannian Geometry of DLN}
\label{sec:dln}
This section contains several results on the state space geometry for the DLN. We first review the balanced manifold as well as the Riemannian submersion that determines the Riemannian metric for the DLN, basing our discussion on past work~\cite{arora2019implicit,Bah2019}. This is followed by computations in SVD coordinates that provide matrix representations of the metric, including the infinite depth limit $N\to \infty$ (equations~\eqref{eq:metricmatrix-a}--\eqref{eq:metricmatrix-c}).
Finally, we prove Theorem~\ref{thm:volumeforms} on the volume forms.

\subsection{Riemannian submersion of the balanced manifold}
\label{subsec:balanced}
The geometry on $GL(d)$ is defined by a Riemannian submersion of the balanced manifold, introduced in~\cite{Bah2019} and illustrated in Figure~\ref{fig:balanced}. This observation effectively reduces the dynamics of the DLN to a space of dimension $d^2$, independent of the depth $N$. 
\begin{figure}[H]
    \centering
    \includegraphics[scale=0.8]{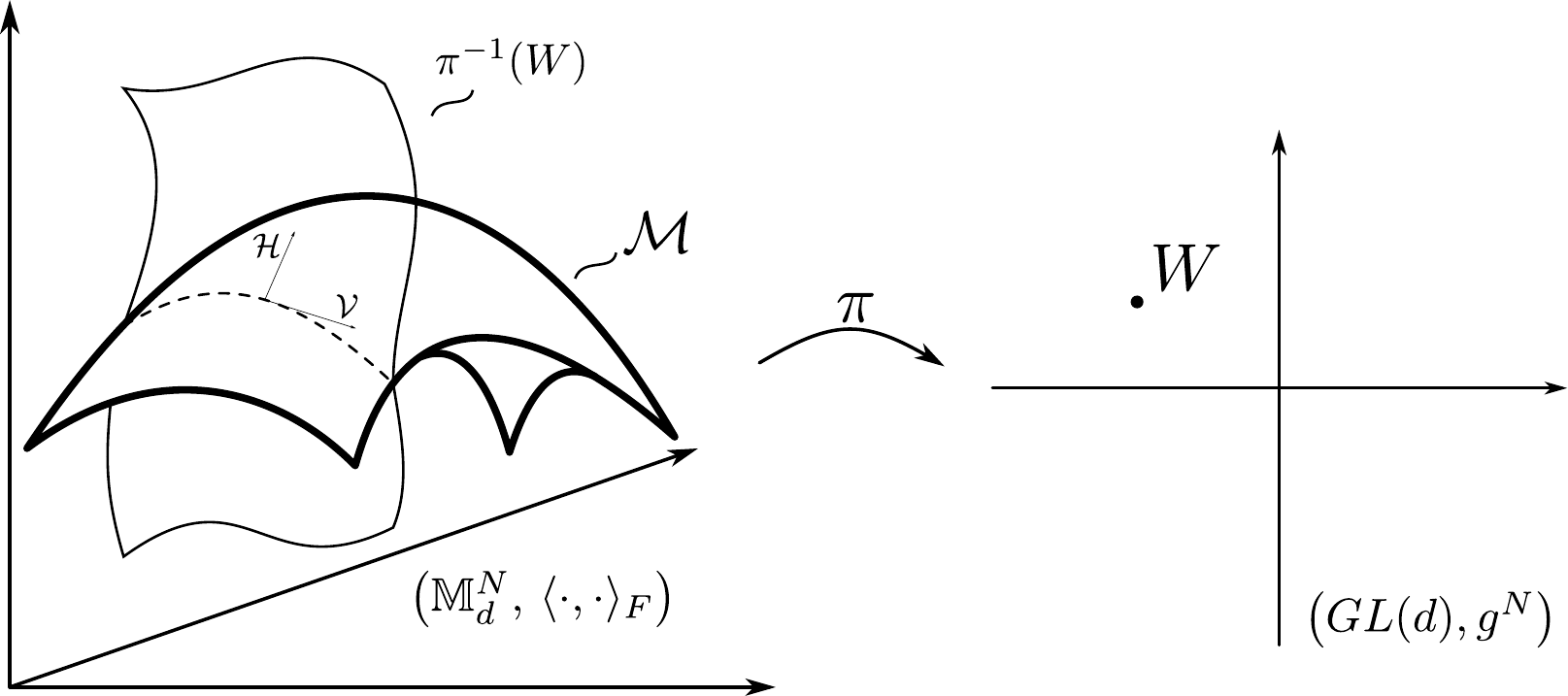}
    \caption{The left side is the optimization space and shows the balanced manifold as an immersed manifold. The right side represents the space of trained networks equipped with a Riemannian metric given by the Riemannian submersion of the balanced manifold.\label{fig:balanced}}
\end{figure}

Recall that $\mathbf{W}=(W_N,W_{N-1},\ldots,W_1) \in \mathbb{M}_d^N$ (see equation~\eqref{eq:def-bigW}).
\begin{definition}[Balanced manifold] The balanced manifold of full rank matrices is 
    \begin{equation}
        \mathcal{M}:=\left\{\mathbf{W}|W_j \in GL(d),\text{ and } G_j=0 \text{ for }j\in1\dots N-1\right\}
        \label{balancedmfd}
    \end{equation}
\end{definition}
For more on the interpretation of balancedness in machine learning, see \cite{Arora2018}. The following lemma is included for completeness; it has already been established in~\cite{Bah2019}.

\begin{lemma} [Invariant manifold] The balanced manifold is invariant under the gradient flow~\eqref{eq:gradflow1}.
\end{lemma}

\begin{proof}
    The invariance of $G_j$ (not just $G_j=0$) along trajectories can be checked by direct computation of the derivative along trajectories of \eqref{eq:gradflow1}:
        \begin{equation}
        \dot G_j= \dot W_{j+1}^T W_{j+1}+ W_{j+1}^T \dot W_{j+1} - \dot W_{j}  W_{j}^T -W_{j} \dot W_{j}^T=0.
    \end{equation}

    We need to show that (\ref{balancedmfd}) implicitly defines a manifold. For that we need to check if its differential is surjective as a linear map from $\mathbb{M}_{d} \times \mathbb{M}_{d}$ 
    to the space of $d\times d$ symmetric matrices. In other words we need to show that for $W_j, W_{j+1} \in GL(d)$ and arbitrary symmetric $S$ the equation
    \begin{equation}
        \dot W_{j+1}^T W_{j+1}+W_{j+1}^T \dot W_{j+1}-\dot W_j W_j^T-W_j \dot W_j^T=S
    \end{equation}
    has a solution for $\dot W_j, \dot W_{j+1}$. It is easy to see that $\dot W_{j+1}^T W_{j+1}+W_{j+1}^T \dot W_{j+1}$ and $-\dot W_j W_j^T-W_j \dot W_j^T$ are arbitrary symmetric matrices. Thus the problem reduces to writing $S$ as a sum of two symmetric matrices.
\end{proof}

The invariance of $G_j$ in the above proof does not require that $G_j=0$, $1 \leq j \leq N$. In practice, $\mathcal{M}$ is particularly important because the DLN is initialized with small initial conditions. Thus, all singular values are small, and the dynamical system begins close to $\mathcal{M}$. 

\begin{lemma} (Symmetries)
\label{lem:symmetry}
    For $Q \in O(d)$, let $L_i(Q)$ denote the linear map
    \begin{equation}
        L_i(Q)(\mathbf{W})=\left(W_N, W_{N-1},\dots,W_{i+1}Q,Q^TW_{i},\dots, W_1\right). 
    \end{equation}
    Then $\pi\left(L_i\left(Q\right)\left(\mathbf{W}\right)\right)=\pi(\mathbf{W})$ and $L_i\left(Q\right)\left(\mathcal{M}\right)=\mathcal{M}$.
\end{lemma}

\begin{proof}
    Both statements are obtained through direct computations. In order to see that $\pi\left(L_i\left(Q\right)\left(\mathbf{W}\right)\right)=\pi\left(\mathbf{W}\right)$, we compute
    \[ W_NW_{N-1}\cdots W_{i+1}Q^T Q W_i \cdots W_1 = W_NW_{N-1}\cdots W_{i+1}W_i \cdots W_1,\]
    since $QQ^T=I$. 
    
    Next assume $\mathbf{W}$ is balanced and observe that under the action of $L_i(Q)$
    \begin{equation}
         W_{i+1} Q Q^T W_{i+1}^T=W_{i+1} W_{i+1}^T= W_{i+2}^TW_{i+2}  ,
    \end{equation}
    \begin{equation}
          Q^T W_{i+1}^T  W_{i+1} Q=Q^T W_{i}  W_{i}^T Q
    \end{equation}
    and
    \begin{equation}
         W_{i}^T Q Q^T  W_{i} Q=W_{i}^T   W_{i}= W_{i-1}  W_{i-1}^T.
    \end{equation}

\end{proof}



\subsection{The operator $\mathcal{A}_{N,W}$ and the metric $g^N$}
\label{subsec:metric}
The metric $g^N$ was defined in~\cite{Bah2019} using the linear operator $\mathcal{A}_{N,W}: T_WGL(d) \to \mathbb{M}_{d}$ (see equations~\eqref{eq:pushforward1}--\eqref{metric} and~\cite[Defn.3]{Bah2019}). We find it convenient to represent $\mathcal{A}_N$ and $g^N$ in SVD coordinates since this yields explicit formulas for the metric and volume form. In all that follows we write the SVD of a matrix $W$ as 
\begin{equation}
    \label{eq:svd-notation}
    W = U\Sigma V^T, \quad \Sigma = \mathrm{diag}(\sigma_1, \ldots, \sigma_d), \quad \sigma_1 \geq \sigma_2 \geq \ldots \geq \sigma_d.
\end{equation}
The standard basis of $\mathbb{M}_d$ is denoted by $E_{ij}$, $1\leq i,j\leq d$.
\begin{lemma}[Spectral decomposition of $\mathcal{A}_{N,W}$]
\label{diagonalization}
   For finite $N$, the $d^2$ number of eigenvalues of $\mathcal{A}_{N,W}$ are 
    \begin{equation}
    \label{eq:eig1}
        \lambda^N_{il}=\frac{1}{N}\sum_{j=1}^N \left(\sigma_i^2\right)^\frac{N-j}{N} \left(\sigma^2_l\right)^\frac{j-1}{N}=\frac{\left(\sigma_i^2\right)^{\frac{N-1}{N}}}{N} \frac{1-\frac{\sigma_l^2}{\sigma_i^2}}{1-\left(\frac{\sigma_l^2}{\sigma_i^2}\right)^\frac{1}{N}},\,\, i,l \in 1,\dots,d.
    \end{equation}
    In the limit $N=\infty$, the eigenvalues are 
    \begin{equation}
    \label{eq:eig2}
        \lambda^\infty_{il}=\frac{\sigma_i^2-\sigma_l^2}{2\log \left(\sigma_i/\sigma_l\right)},\, i,l \in 1,\dots,d.
    \end{equation}
    The corresponding eigenvectors are independent of $N$ and are given by
    \begin{equation}
        T_{il}=U E_{il} V^T, \quad 1\leq i,l \leq d.
    \end{equation}
\end{lemma}

\begin{proof}
      Let
      \begin{equation}
        Y=\mathcal{A}_{N}(X),
      \end{equation}
      and introduce new variables $\tilde{X}=U^T X V$ and $\tilde{Y}=U^T Y V$. By the definition of $\mathcal{A}_{N}$
      \begin{equation}
        \tilde{Y}=\frac{1}{N}\sum_{j=1}^N \left(\Sigma^2\right)^\frac{N-j}{N} \tilde{X} \left(\Sigma^2\right)^\frac{j-1}{N}
      \end{equation}
    Notice that we have obtained a diagonal form in coordinates:
      \begin{equation}
      \label{eq:tilde-y}
          \tilde{y}_{il}=\tilde{x}_{il}\frac{1}{N}\sum_{j=1}^N \left(\sigma_i^2\right)^\frac{N-j}{N} \left(\sigma^2_l\right)^\frac{j-1}{N},
        \end{equation}
    which proves equation~\eqref{eq:eig1}. 
    
    The proof of equation~\eqref{eq:eig2} is similar. Fix $i \neq l$ and take the limit $N\to \infty$ in equation~\eqref{eq:tilde-y} to obtain
    \begin{equation}
        \lim_{N\rightarrow \infty}\lambda^N_{il}=\lim_{N\rightarrow \infty}\frac{1}{N}\sum_{j=1}^N \left(\sigma_i^2\right)^\frac{N-j}{N} \left(\sigma^2_l\right)^\frac{j-1}{N} =  \int_0^1 \sigma_i^{2(1-t)} \sigma_l^{2t} dt= \frac{\sigma_i^2-\sigma_l^2}{\log \left(\sigma_i^2/\sigma_l^2\right)}.
      \end{equation}

      Finally, when  $i=l$
      \begin{equation}
        \lim_{N\rightarrow \infty} \lambda^N_{il} = \lim_{N\rightarrow \infty} \left(\sigma_i^2\right)^\frac{N-1}{N} = \sigma_i^2.
      \end{equation}
\end{proof}

As a direct consequence of Lemma~\ref{diagonalization}, we obtain an explicit matrix representation for the DLN metric \eqref{metric}. Let the vectorization of the matrix coordinate one-forms be denoted $dw^\alpha$, $\alpha=1,\dots,d^2$. Then the metric  $g^N=g^N_{\alpha \beta} dw^\alpha \otimes dw^\beta$, $N \leq \infty$ at a point $W=U \Sigma V^T$  is given by
    \begin{equation}
        g^N = \left(V\otimes U\right)D^N\left(\Sigma\right)\left(V\otimes U\right)^T, \label{metricmatrix}
    \end{equation}
    where $D^N$ is the diagonal operator whose non-zero entries are
\begin{equation}
\label{eq:dn-il}
    \frac{1}{\lambda_{il}^N}, \quad 1 \leq i,l \leq d.
\end{equation}  
The above calculations prove equations~\eqref{eq:metricmatrix-a}--\eqref{eq:metricmatrix-c}.



The following inequalities are used to establish normal hyperbolicity of an invariant manifold in Theorem~\ref{thm:nhim} below.
    \begin{lemma}[Inequalities on the spectrum of $\mathcal{A}_{\infty,W}$]
        Let $1\leq i<j\leq k<l\leq d$, then $\lambda^\infty_{jj} \leq \lambda^\infty_{ij} \leq\lambda^\infty_{ii}$ and $\lambda^\infty_{ij} \leq\lambda^\infty_{kl}$.
        \label{lem:spectrumordering}
    \end{lemma}
    
    \begin{proof}
        $\lambda^\infty_{jj}=\sigma_j^2 \leq\sigma_i^2=\lambda^\infty_{ii}$ since $\sigma_j \leq\sigma_i$.
        
        Given a concave differentiable function $f$ and $x<y$,
        \begin{equation}
            f'(y) \leq\frac{f(y)-f(x)}{y-x} \leq f'(x).
        \end{equation}
        Thus
        \begin{equation}
            \frac{d}{dx} \left.\log(x)\right\vert_{\sigma_j^2}=\frac{1}{\sigma_j^2}\geq \frac{\log \sigma_j^2 -\log \sigma_i^2}{\sigma_j^2-\sigma_i^2} \geq \frac{d}{dx} \left.\log(x)\right\vert_{\sigma_i^2}=\frac{1}{\sigma_i^2},
        \end{equation}
        where the middle term is
        \begin{equation}
            \frac{\log \sigma_j^2 -\log \sigma_i^2}{\sigma_j^2-\sigma_i^2} = \frac{1}{\lambda^\infty_{ij}}.
        \end{equation}
        The remaining inequality also follows from the concavity of $\log(x)$.
    \end{proof}

 \subsection{Volume forms and the proof of Theorem \ref{thm:volumeforms}}
 Vandermonde determinants appear often in random matrix theory as the Jacobians for diagonalization (see for example~\cite[\S 5.3]{Deift}, \cite[\S 2.2]{Govind_RMT} and Lemma~\ref{svdlemma} below). Theorem~\ref{thm:volumeforms} reflects an analogous feature of the DLN geometry. The proof is an easy consequence of Lemma~\ref{diagonalization}.
   
      \begin{proof}[Proof of Theorem \ref{thm:volumeforms}]
      We compute the determinants of the matrices $g^N$ and $g^\infty$ as a product of the  reciprocal of the eigenvalues of $\mathcal{A}_{N,W}$ given in Lemma~\ref{diagonalization}. For $g^\infty$ we find
    \begin{equation}
      \begin{split}
        \sqrt{\det g^\infty} dW=\sqrt{\frac{1}{\sigma_1^2 \dots \sigma_d^2} \prod_{i\neq j} \frac{\log\left(\sigma_i^2\right)-\log\left(\sigma_j^2\right)}{\sigma_i^2-\sigma_j^2}} dW
      \\=
        \frac{1}{\sqrt{\det\left(\Sigma^2\right)}} \prod_{i< j} \frac{\log\left(\sigma_i^2\right)-\log\left(\sigma_j^2\right)}{\sigma_i^2-\sigma_j^2} dW
      \\=
        \frac{\mathrm{van}\left(\log \Sigma^2\right)}{\sqrt{\det\left(\Sigma^2\right)} \mathrm{van} \left(\Sigma^2\right)}  dW.
      \end{split}
      \label{eq:volumedW}
      \end{equation}
      
      Similarly, for $N<\infty$, when $i \neq l$ we use the eigenvalues of $\mathcal{A}_N$ to find

      \begin{equation}
        \begin{split}
        \lambda^N_{il}=\frac{1}{N}\sum_{j=1}^N \left(\sigma_i^2\right)^\frac{N-j}{N} \left(\sigma^2_l\right)^\frac{j-1}{N}= \frac{\left(\sigma_i^2\right)^{\frac{N-1}{N}}}{N}
       \left(1+\left(\frac{\sigma_l^2}{\sigma_i^2}\right)^\frac{1}{N}+\dots+\left(\frac{\sigma_l^2}{\sigma_i^2}\right)^\frac{N-1}{N}\right)
        \\=\frac{\left(\sigma_i^2\right)^{\frac{N-1}{N}}}{N} \frac{1-\frac{\sigma_l^2}{\sigma_i^2}}{1-\left(\frac{\sigma_l^2}{\sigma_i^2}\right)^\frac{1}{N}}.
      \end{split}
      \end{equation}

      In the case $i = l$ we obtain instead
      \begin{equation}
        \lambda^N_{il} = \left(\sigma_i^2\right)^\frac{N-1}{N}.
      \end{equation}

      The volume form is given by the product of the reciprocals of all the eigenvalues $\lambda_{il}^N$, $1 \leq i,l \leq d$. Thus, 

        \begin{equation}
          \begin{split}
            \sqrt{\det g^N}dW&=\sqrt{\prod_{i=1}^d\frac{1}{\left(\sigma_i^2\right)^\frac{N-1}{N}} \prod_{i\neq l} N\left(\sigma_i^2\right)^{\frac{1-N}{N}} \frac{1-\left(\frac{\sigma_l^2}{\sigma_i^2}\right)^\frac{1}{N}}{1-\frac{\sigma_l^2}{\sigma_i^2}}} dW
            \\&=\frac{N^\frac{d(d-1)}{2}\det(\Sigma^2)^{\frac{1-N}{2N}}}{V(\Sigma^2)}V\left(\Sigma^{2/N}\right) dW.
          \end{split}
        \end{equation}
Finally, we use Lemma~\ref{svdlemma} below to express $dW$ in SVD coordinates completing the proof of Theorem~\ref{thm:volumeforms}.
      
      \end{proof}
\begin{remark}[Volume form in new coordinates]
        Notice that using the coordinates $\Lambda=\log \left( \Sigma\right)$, we have $\frac{d\lambda_i}{d\sigma_i}=\frac{1}{\sigma_i}$ and thus
        \begin{equation}
            \sqrt{\det g^\infty} dW= 2^\frac{d(d-1)}{2} \text{van} \Lambda\, d\Lambda dU dV.
        \end{equation}

        This shows a strong formal similarity between the DLN and the Gaussian Orthogonal Ensemble of random matrix theory, suggesting the study of probability distributions and large $d$ asymptotics.
    \end{remark}
    
    \subsection{The Jacobian of singular value decomposition}
    The following Lemma is included for completeness since we were unable to find a convenient reference.

      \begin{lemma}
      \label{svdlemma}
        The Jacobian determinant of the SVD map: $W \mapsto (U,\Sigma,V) \in O(d) \times \mathbb{R}^d\times O(d)$ is given by the Vandermonde determinant $ \mathrm{van}(\Sigma^2)$.
      \end{lemma}
      \begin{proof}
      Assume $W$ is a matrix with distinct singular values. Consider a $C^1$ curve $W(t)$, $t \in [-1,1]$, with $W(0)=W$ such that the SVD coordinates of $W(t)$, written $U(t)\Sigma(t)V(t)^T$ are also $C^1$. Let $\dot{W}(0)$ be denoted $\dot{W}$ and similarly for $U$,$\Sigma$ and $V$. We compute
                 \begin{equation}
                  \dot W=\dot U\Sigma V^T + U \dot \Sigma V^T + U\Sigma \dot V^T.
                \end{equation}
              Since $U,V$ are orthogonal and $\Sigma$ is diagonal,
              \[ \dot{U}=UA^U, \quad \dot{V}=V A^V, \quad \dot{\Sigma}=D\]
              where $A^U$ and $A^V$ are skew-symmetric and $D$ is diagonal. Thus, we may write
              \begin{equation}
              \dot W=U(A^U \Sigma + D + \Sigma A^V)V^T.
              \end{equation}
            This expression defines $\dot{W}$ as the image of a linear map from $TO(d)\times TO(d) \times T\mathbb{R}^d$ to $TGL(d)$. The Jacobian we are seeking is the determinant of this map.  We compute this determinant, by first simplifying the expression above with the isometry $\dot{W}\mapsto U^T\dot{W}V$, and defining the linear transformation
            \begin{equation}
            \mathcal{L}: TO(d)\times TO(d) \times T\mathbb{R}^d \to \mathbb{R}^{d\times d}, \quad (A^U,A^V,D) \mapsto A^U \Sigma + D + \Sigma A^V.
              \end{equation}
           Finally, we compute $\det(\mathcal{L})$ as follows.  First observe that the action of $\mathcal{L}$ on diagonal matrices has eigenvalue $1$ with multiplicity $d$. The action of $\mathcal{L}$ on $TO(d)\times TO(d)$ is given in coordinates by
                \begin{equation}
                  \sum_j a_{ij}^U \delta_{jk} \sigma_j + \delta_{ij} \sigma_i a_{jk}^V.
                \end{equation}

                Introducing the row-major half-vectorization for the skew-symmetric matrices (example for $d=3$: $\text{vec}(A^U)=[a_{12}^U\, a_{13}^U\, a^U_{23}]^T$) the action of $\mathcal{L}$ is now represented by the block matrix with diagonal blocks:
                \begin{equation}
                  \begin{bmatrix}
                    D_1 & D_2 \\ -D_2 & -D_1
                  \end{bmatrix}
                  \begin{bmatrix}
                    \text{vec}(A^U) \\ \text{vec}(A^V)
                  \end{bmatrix},
                  \label{eq:blockmatrix}
                \end{equation}
                where $D_1$ is
                \begin{equation}
                  D_1=\text{diag}(\underbrace{\sigma_1,\dots,\sigma_1}_{d-1}, \underbrace{\sigma_2,\dots,\sigma_2}_{d-2},\dots, \underbrace{\sigma_{d-1}}_1),
                \end{equation}
                and $D_2$ is
                \begin{equation}
                  D_2=\text{diag}(\underbrace{\sigma_2,\dots,\sigma_d}_{d-1}, \underbrace{\sigma_3,\dots,\sigma_d}_{d-2},\dots, \underbrace{\sigma_{d-1},\sigma_{d}}_{2}, \underbrace{\sigma_{d}}_{1}).
                \end{equation}
                The matrices $D_1$ and $D_2$ commute because they are diagonal. Thus, the absolute value of the determinant of the block matrix above is
                \begin{equation}
                  |\det(D_2^2-D_1^2)|=|\det\text{diag}(\sigma_i^2-\sigma_j^2,\,i>j)|=\prod_{i<j}(\sigma_i^2-\sigma_j^2)=\mathrm{van}(\Sigma^2).
                \end{equation}

      \end{proof}

\section{Dynamics of matrix completion}
\label{sec:matrixcompletion}    
This section is devoted to the gradient dynamics of the DLN for matrix completion. We consider quadratic energy functions as in equation~\eqref{loss} and equation~\eqref{eq:diagonalenergy}. As $\mathcal{B}$ varies, the energy $E_\mathcal{B}$ may have a unique minimum or a submanifold of minima. By equation~\eqref{eq:gradient-flow} the dynamics of the DLN is determined by an interplay between the Riemannian geometry and the nature of $E_{\mathcal{B}}(W)$. 

This section primarily focuses on describing numerical experiments. However, in order to generate efficial numerical simulations, it is necessary to first express the dynamics in SVD coordinates. These expressions are summarized in Theorem~\ref{svddynamicsgeneral}. We also prove Theorem~\ref{thm:nhim} on normal hyperbolicity of submanifolds of equilibria to shed light on the attraction to the balanced manifold.

The main themes in the numerical experiments are as follows. First, we consider the variation with depth $N$, in order to demonstrate the utility of the infinite depth limit. Second, we construct energies where low rank heuristics are insufficient to explain the accumulation of training outcomes. Instead, we demonstrate that state space volume (as measured by the intrinsic Riemannian metric $g^N$) is a better predictor of the training outcome. 

\subsection{Numerical integration of equation~\eqref{eq:gradient-flow} using SVD}
\label{subsec:svddynamics}

Numerically integrating \eqref{eq:gradient-flow} gets very costly with increasing depth, due to the large number of matrix powers needed. Alternatively, one can use the factorized formula \eqref{metricmatrix} to compute the Riemannian gradient through the dual metric:
\begin{equation}
\label{eq:generalgradientflow}
    \dot w=-g^{N*}(w) \partial E(w).
\end{equation}

This formulation, however, requires the SVD of $W$ at every time step.  Instead, we compute the SVD of the initial condition and directly evolve the singular coordinates using smooth singular value decomposition. All numerical results in this paper are obtained using the fourth order, fixed timestep Runge-Kutta method. This formulation is stated in Theorem \ref{svddynamicsgeneral}, building on the well known result Lemma~\ref{smooth_svd}, that we review for completeness.

  \begin{lemma}[Smooth SVD]
    \label{smooth_svd}
      Given a smooth curve $W(t): (t_1,t_2)\rightarrow GL(d)$, $W(t)$ having distinct singular values for all $t \in (t_1,t_2)$, a smooth singular value decomposition $W(t)=U(t) \Sigma(t) V(t)^T$ exists satisfying the following system of differential equations:
      \begin{equation}
        \dot \sigma_i=u_i^T \dot W v_i
        \label{eq:svperturbation}
      \end{equation}

      \begin{equation}
        \dot u_i =\sum_{j\neq i} \frac{1}{\sigma^2_i-\sigma^2_j} \langle (\dot W W^T +W \dot W^T) u_i, u_j \rangle u_j
      \end{equation}

      \begin{equation}
        \dot v_i =\sum_{j\neq i} \frac{1}{\sigma^2_i-\sigma^2_j} \langle (\dot W^T W +W^T \dot W) v_i, v_j \rangle v_j
      \end{equation}
    \end{lemma}

    \begin{proof}
        Under our assumptions these formulas can be verified by direct computation. For a more careful treatment for dealing with repeated singular values see \cite{analyticsvd}.
    \end{proof}
    
    Using Lemma \ref{smooth_svd} and equation \eqref{eq:generalgradientflow}, we can write down the evolution equations for the singular coordinates $U(t) \Sigma(t) V^T(t)=W(t)$ directly. For $N<\infty$, this result is equivalent to statements in \cite{arora2019implicit}(see Thm 3 and Lemma 2). Let $\mathfrak{s}\left(M\right):=M-M^T$, $\alpha=1-1/N$ and $\circ$ denote the Hadamard (elementwise) product. For $N \leq \infty$ introduce the matrix:
    \begin{equation}
        L_N^{il}=
        \begin{cases}
            \frac{\lambda^N_{il}}{\sigma_i^2-\sigma_l^2},\text{ for }i\neq l\\
            0\text{ otherwise.}
        \end{cases}
    \end{equation}

    \begin{theorem}\label{svddynamicsgeneral}
        Under the assumptions of Lemma \ref{smooth_svd}, and differentiable loss function $E$, the singular value decomposition of the end-to-end matrix evolve according to the differential equations:
     \begin{align}
        \dot U &= U \mathfrak{s}\left(\left(L_N\left(\Sigma\right) \Sigma\right)\circ\left(U^T \partial_W E\, V\right)\right)\\
        \dot \Sigma &= -\Sigma^{2\alpha} \mathrm{diag} \left(U^T \partial_W E\, V\right)\label{eq:svdynamics}\\
        \dot V &= V \mathfrak{s}\left(\left( \Sigma L_N\left(\Sigma\right)\right)\circ\left(U^T \partial_W E\, V\right)\right).
    \end{align}
    \end{theorem}
    
    \begin{proof}
       For $N<\infty$, see Theorem 3 and Lemma 2 in \cite{arora2019implicit}. Under infinite depth, direct computation using \ref{eq:generalgradientflow} and Lemma \ref{smooth_svd} verifies the result.
    \end{proof}

    Recall that the state space of the Riemannian gradient flow is $GL(d)$. Since this is a dense, open subset of the space of all quadratic matrices of given size, any lower rank matrix can be approximated in this space. An appropriate way to quantify the rank of a matrix that is stable under numerical approximations is the following.

    \begin{definition}[Effective rank]
        Let $W \in GL(d)$ with singular values $\sigma_i$ and $s_i$ be its normalized singular values
        \begin{equation}
            s_i=\frac{\sigma_i}{\sum_j \sigma_j}.
        \end{equation}
        Then the effective rank of $W$ is
        \begin{equation}
            r_e\left(W\right)=\exp\left({-\sum_i s_i \log\left(s_i\right)}\right).
        \end{equation}
    \end{definition}

\subsection{Attraction rates and the proof of Theorem \ref{thm:nhim}}
\label{sec:attraction}

We derive bounds on the normal attraction rates for the submanifold of equilibria $\mathcal{N}_\mathcal{B}$. Given a choice of $\mathcal{B}$, let $\mathcal{I}$ denote the vectorized index set of observed elements, that is
\begin{equation}
    \text{vec} \left( \mathcal{B}\right)_i = \begin{cases}
        1 \text{ if } i\in \mathcal{I} \\ 0 \text{ otherwise.}
    \end{cases}
\end{equation}

In the following, we compute the linearization of equation~\eqref{eq:gradient-flow} in the normal direction of $\mathcal{B}_\mathcal{N}$. We use coordinates $W=\mathcal{B} \circ \Phi +X+Y$, where $\mathcal{B} \circ X=X$ and $\mathcal{B} \circ Y=0$. Lowercase $x,y$ and $w$ denote the vectorization of these coordinates. The dual metric tensor (the matrix of which is represented by the inverse of $g^{N}$) is denoted $g^{N*}$. Let $W_0 \in \mathcal{N}_\mathcal{B}$, be a fixed point.

\begin{lemma}
\label{lem:rates}
At $W_0$, the linearization of equation~\eqref{eq:gradient-flow} in the normal direction $x$ is given by 
\begin{equation}
    \dot x= -A    x,
\end{equation}
where $A$ is a principle submatrix of $g^{N*}$, consisting of rows and columns of indices observed by $\mathcal{B}$:
\begin{equation}
    A=\left[\left. g^{N*}  \right\vert_{W_0} \right]_{i,j \in \mathcal{I}}.
\end{equation}

\end{lemma}

\begin{proof}
    Using the general form of the Riemannian gradient flow \eqref{eq:generalgradientflow} and taking a derivative
  \begin{equation}
      \frac{d}{d\epsilon} g^{N*}(\epsilon x,y) \partial E_I((\epsilon x,y))=\left(\partial_x g^{N*}\right) \partial E_I(w_0)+g^{N*}(w_0) \partial^2 E(w_0) x,
  \end{equation}
  we notice that the first term is zero since $\partial_W E(W_0)=0$ by $W_0$ being an equilibrium of the gradient flow. Computing
  \begin{equation}
      \partial^2 E_i(W_0)_{ij}=\begin{cases}
          1, \text{ if }i=j, i \in \mathcal{I}\\ 0 \text{ otherwise,}
      \end{cases}
  \end{equation}
  we see that multiplication by this matrix results in the above principle submatrix of $g^{N*}$.
\end{proof}

\begin{proof}[Proof of Theorem \ref{thm:nhim}]
     We show a nonzero lower bound on the attraction rates. Let $\{\alpha_i\}_1^d$, $\alpha_i\leq \alpha_j$ if $j\leq i$ denote the eigenvalues of $A$. Then by the Cauchy interlacing theorem and Lemma \ref{lem:spectrumordering},
    \begin{equation}
        \sigma_d^2\leq\alpha_d\leq \alpha_{d-1} \leq \dots \leq \alpha_1,
    \end{equation}
    and note that this lower bound is nonzero on any compact subset on $\mathcal{N}_\mathcal{B}$.
\end{proof}

Notice that for Lemma \ref{lem:spectrumordering} and the Cauchy interlacing theorem completely characterizes the order of $\lambda^\infty_{ij}$ and the characteristic exponents of the $d=2$ case:
\begin{equation}
    \sigma_2^2 \leq \alpha_2 \leq\frac{\sigma_1^2-\sigma_2^2}{\log \sigma_1^2-\log \sigma_2^2} \leq \alpha_1  \leq \sigma_1^2.
    \label{eq:boundsonrates}
\end{equation}

\subsection{Diagonal matrix completion}

In this section we show numerical simulations for the energy function $E_I$ under variable $N$ and $d$. For this choice of energy function, the rank of possible completions ranges from $1$ to $d$, and so it constitutes one of the important cases for studying bias towards low rank in matrix completion problems.



\subsubsection{Example: $d=20$}
\label{sec:wideexample}

We start with simulations of a larger, $d=20$, example. Figures \ref{depth10Infmtrxcompl}-\ref{depth3mtrxcompl} show histograms of effective rank of optimization outcomes. Note that the rank and thus the effective rank here can be as large as $20$, the size of the matrices and therefore bias towards low rank could mean convergence to a matrix of any effective rank smaller than $20$.

Notice that even the shallowest, $N=3$ case shows strong bias towards low rank completions. In case of $N=10$ and $N=\infty$, the obtained histograms are nearly identical, showing strong bias towards minimal, rank one outcomes. This suggests that building intuition around the dynamics under sufficient depth is possible using small, $d=2$ or $d=3$ examples, in which case the rank deficient cases \eqref{eq:gradflow1} are easier to characterize. 

Approximately 300 optimization outcomes are included for each $N$. Panel (b)  of Figure \ref{depth3mtrxcompl} shows the distribution of effective rank upon initialization. The small random initial conditions are drawn from a Wigner ensemble, specifically they consist of matrices of independent normal entries of mean zero and standard deviation $0.001$, which distribution is denoted $\text{Wigner}(0,0.001)$).

Numerical convergence criterion used for these examples is $E_I(W)< 10^{-6}$.  

\begin{figure}[H]
\begin{subfigure}{.48\textwidth}
  \centering
  \includegraphics[width=1.0\linewidth]{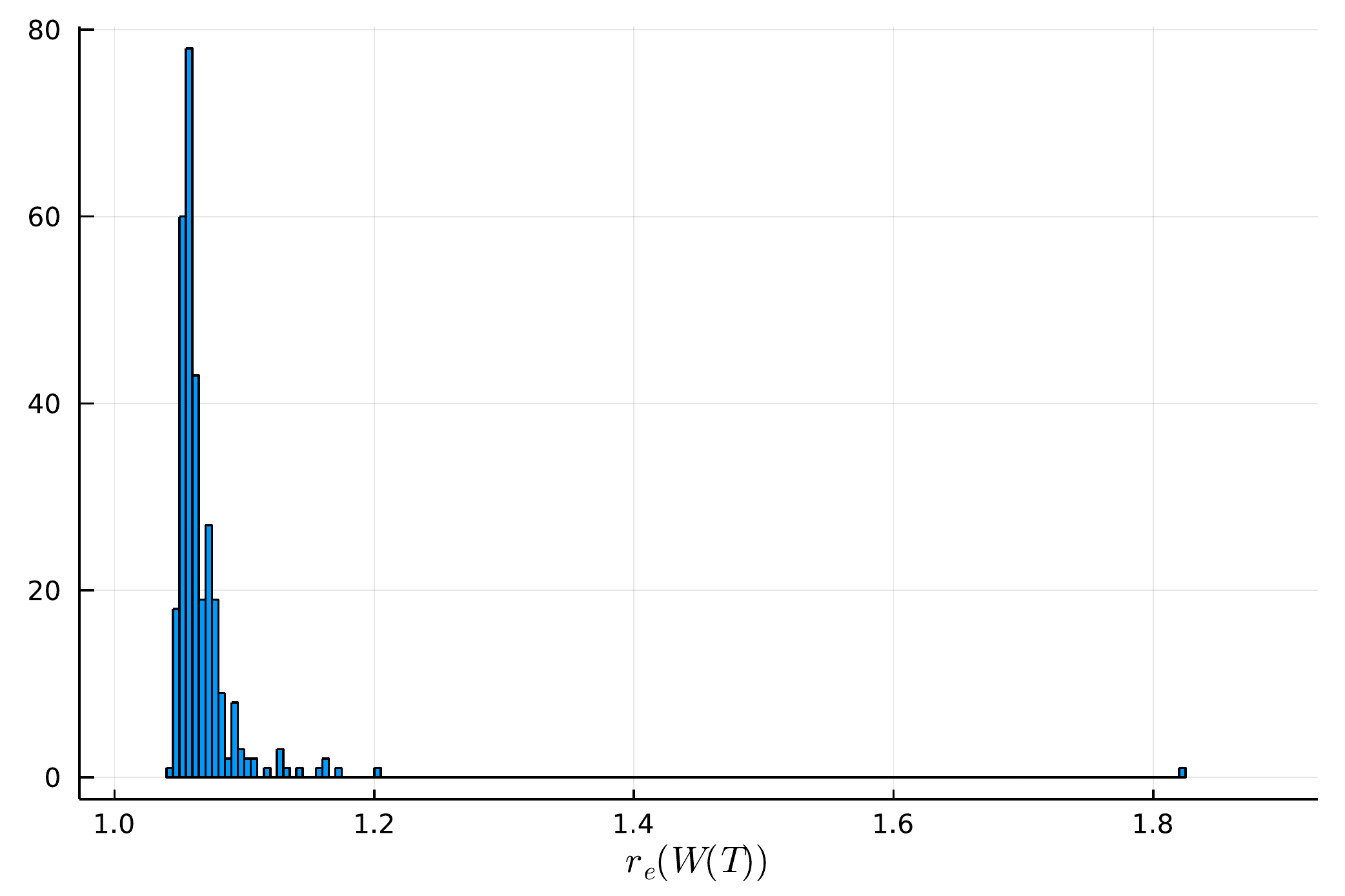}  
  \caption{$N=10$}
  
\end{subfigure}
\hspace{0.02\textwidth}
\begin{subfigure}{.48\textwidth}
  \centering
  \includegraphics[width=1.0\linewidth]{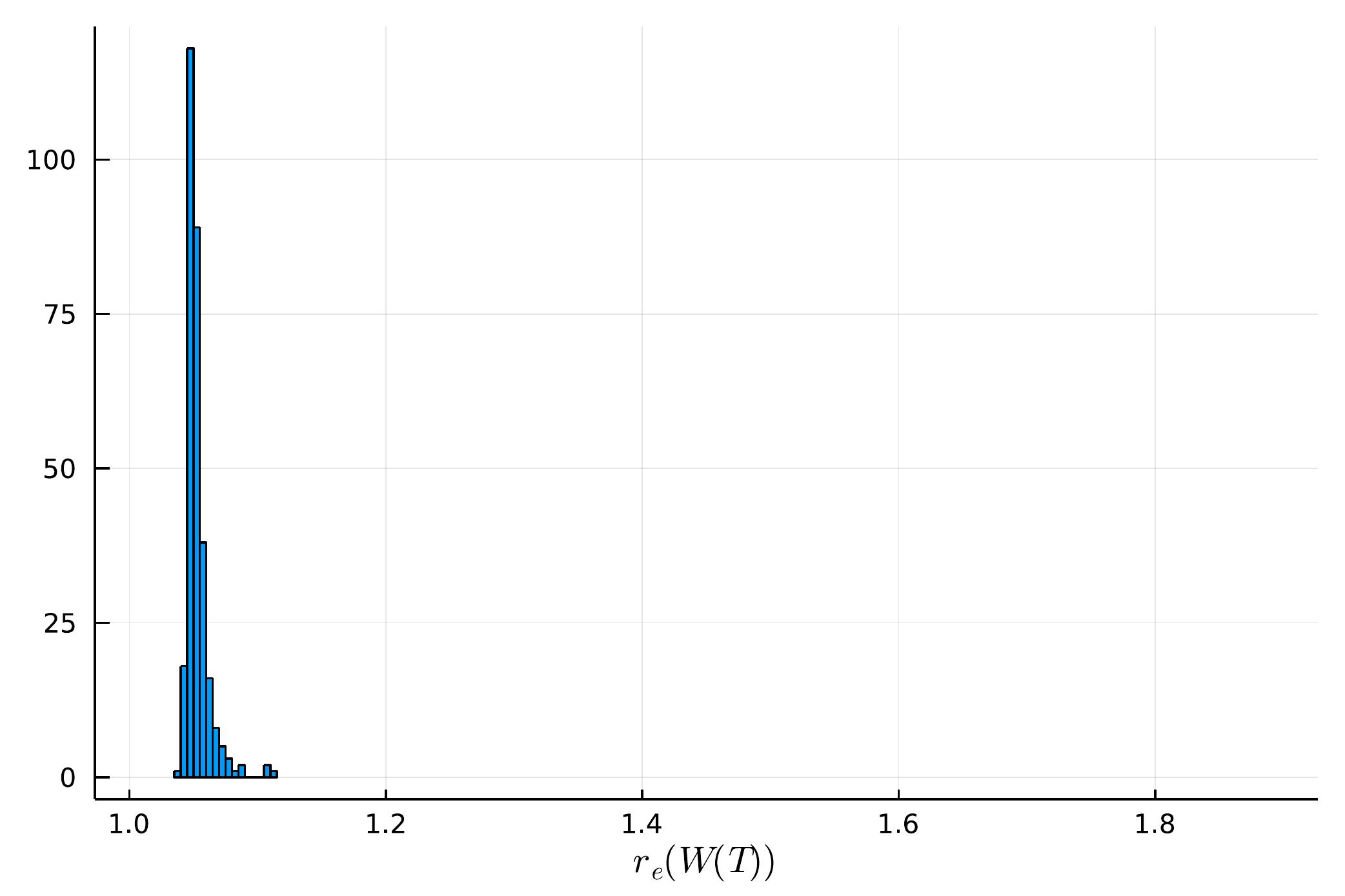}  
  \caption{$N=\infty$}
  
\end{subfigure}
\caption{Empirical distributions of effective rank in $20\times 20$ matrix completion simulations.}
\label{depth10Infmtrxcompl}
\end{figure}

\begin{figure}[H]
\begin{subfigure}{.48\textwidth}
  \centering
  \includegraphics[width=1.0\linewidth]{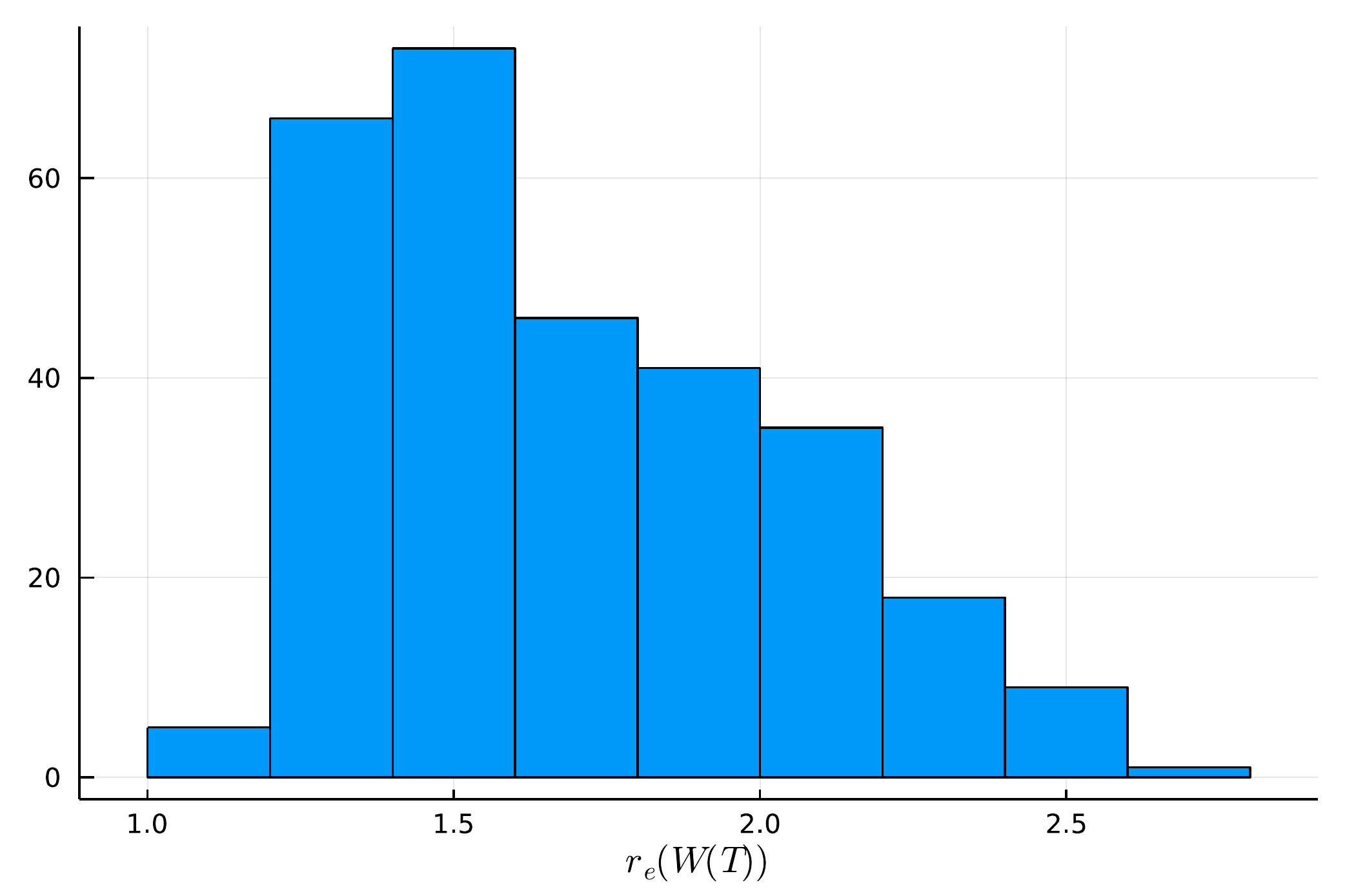}  
  \caption{$N=3$}
  
\end{subfigure}
\hspace{0.02\textwidth}
\begin{subfigure}{.48\textwidth}
  \centering
  \includegraphics[width=1.0\linewidth]{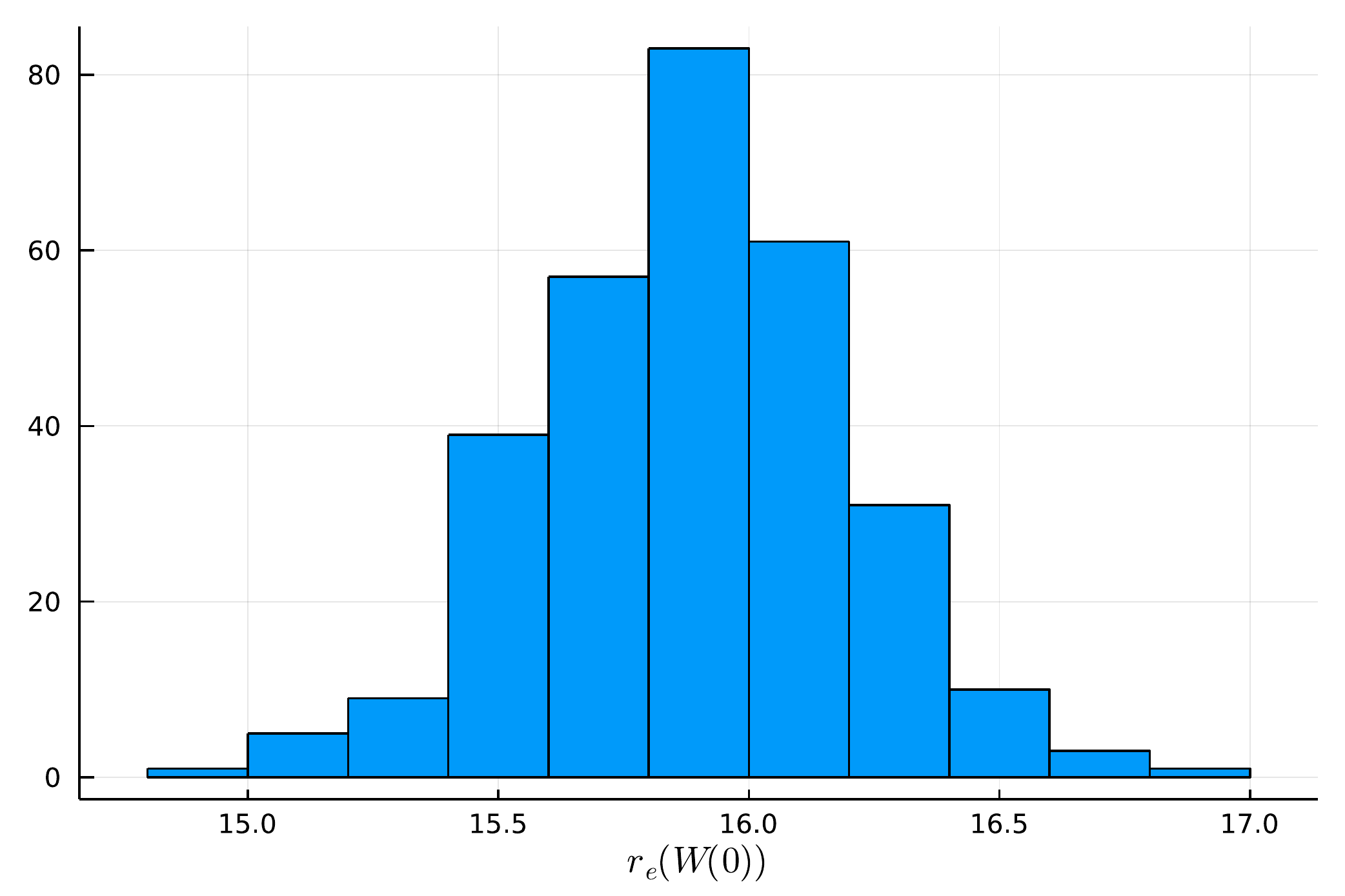}  
  \caption{Initialization}
  
\end{subfigure}
\caption{(a) Effective rank distributions of shallow matrix completion simulations. (b) Sample distribution of effective rank at initialization for all above examples.}
\label{depth3mtrxcompl}
\end{figure}

\subsubsection{Example: $d=2$}
\label{hyperbolas}
In the following example for matrix completion, we illustrate how an increase in depth influences training outcome and at what depth do the results start to closely resemble our infinite depth limit. For easier visualization and building on the intuition developed in the previous section, we keep the width at the minimum $d=2$. The results shown in Figure \ref{fig:diagonalmatrixcompletion5Inf} are obtained by approximately 3000 runs for $N=5,10,20$ and infinity each. Each of these figures show the outcome of these batch simulations. 

For all of these batch runs, the optimization objective is a fixed invertible matrix of diagonal elements $[0.58724; 1.447]$. The small random initial conditions are drawn from $\text{Wigner}(0,0.001)$\footnote{The Wigner ensemble $\text{Wigner}(\mu,\sigma)$ is a distribution of random matrices with independent normal elements of mean $\mu$ and standard deviation $\sigma$.}. Note that initializing the end-to-end matrix directly does not lead to the same distribution as initializing individual layers a similar way and then taking a product. However, for the purposes of this example, these two approaches lead to identical results.

Notice that here any matrix with fixed diagonal elements
\begin{equation}
    \begin{bmatrix}
        \Phi_1 & w_{12} \\ w_{21} & \Phi_2
    \end{bmatrix}
\end{equation}
is a global minimizer. Batch simulations allow us to see whether there is a concentration of training outcomes on the $w_{12},w_{21}$ plane of possible global minimizers. Figure \ref{fig:diagonalmatrixcompletion5Inf} show the results of these simulations, and Figure \ref{fig:logvolumeforms} show the clear correspondence to high phase space volume. The hyperbola on the left pane correspond to the rank one singularities on the manifold defined by
\begin{equation}
    \frac{\Phi_1}{w_{21}}=\frac{w_{12}}{\Phi_2}.
\end{equation}
At these singularities phase space volume blows up for any depth, and the probability of landing in this high volume region is expected to be high. The logarithmic volume can be interpreted as an entropic quantity \cite{Govind_SDP}, suggesting the use of statistical mechanical tools in our future analysis. 

Numerical convergence criterion is $E(W)< 10^{-15}$ which is achieved in no more than $T=1200$ simulation time in all included examples.


\begin{figure}[h!]
\begin{subfigure}{.48\textwidth}
  \centering
  \includegraphics[width=1.0\linewidth]{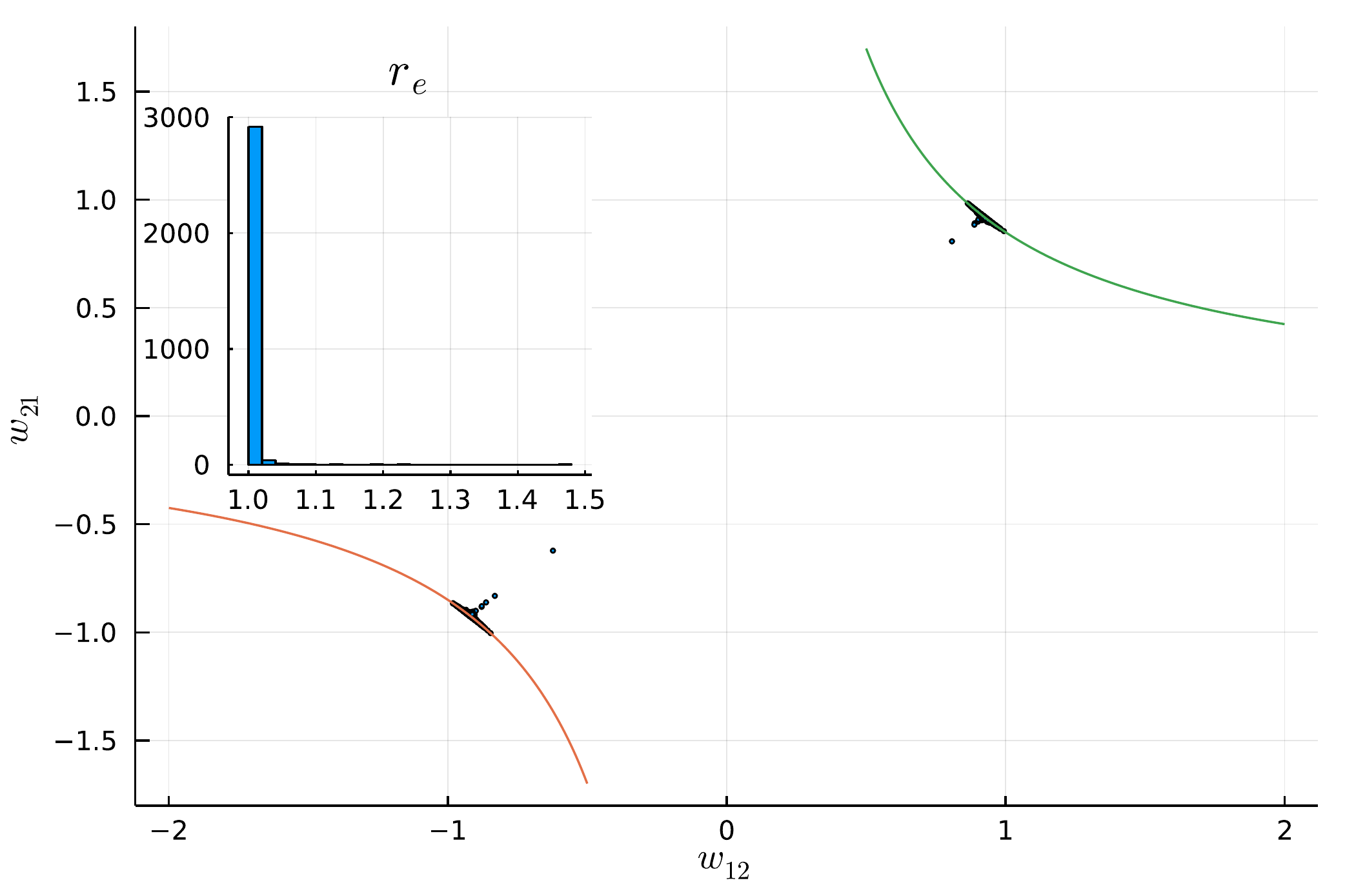}  
  \caption{$N=5$}
  
\end{subfigure}
\hspace{0.02\textwidth}
\begin{subfigure}{.48\textwidth}
  \centering
  \includegraphics[width=1.0\linewidth]{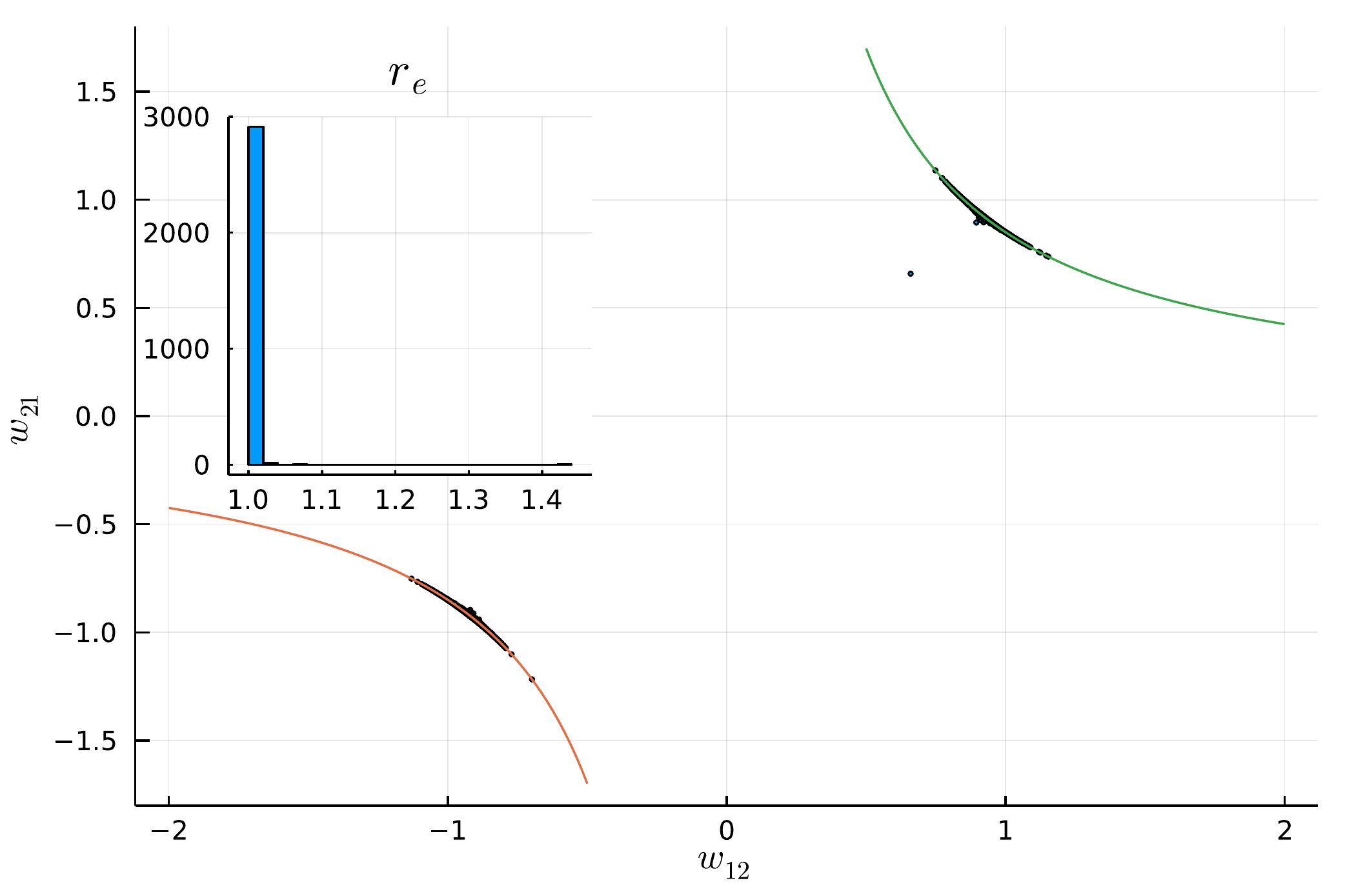}  
  \caption{$N=10$}
  
\end{subfigure}

\begin{subfigure}{.48\textwidth}
  \centering
  \includegraphics[width=1.0\linewidth]{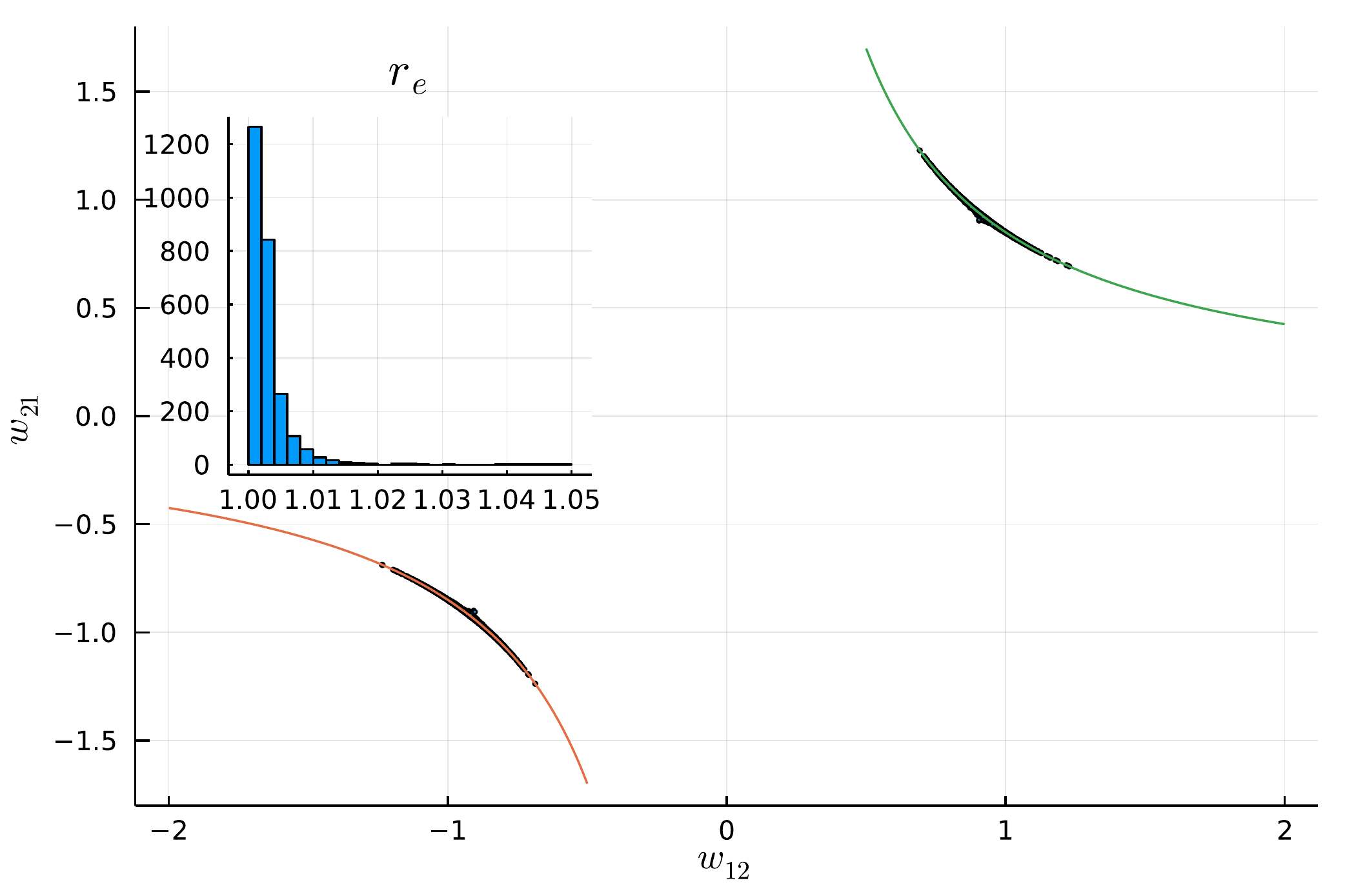}  
  \caption{$N=20$}
  
\end{subfigure}
\hspace{0.02\textwidth}
\begin{subfigure}{.48\textwidth}
  \centering
  \includegraphics[width=1.0\linewidth]{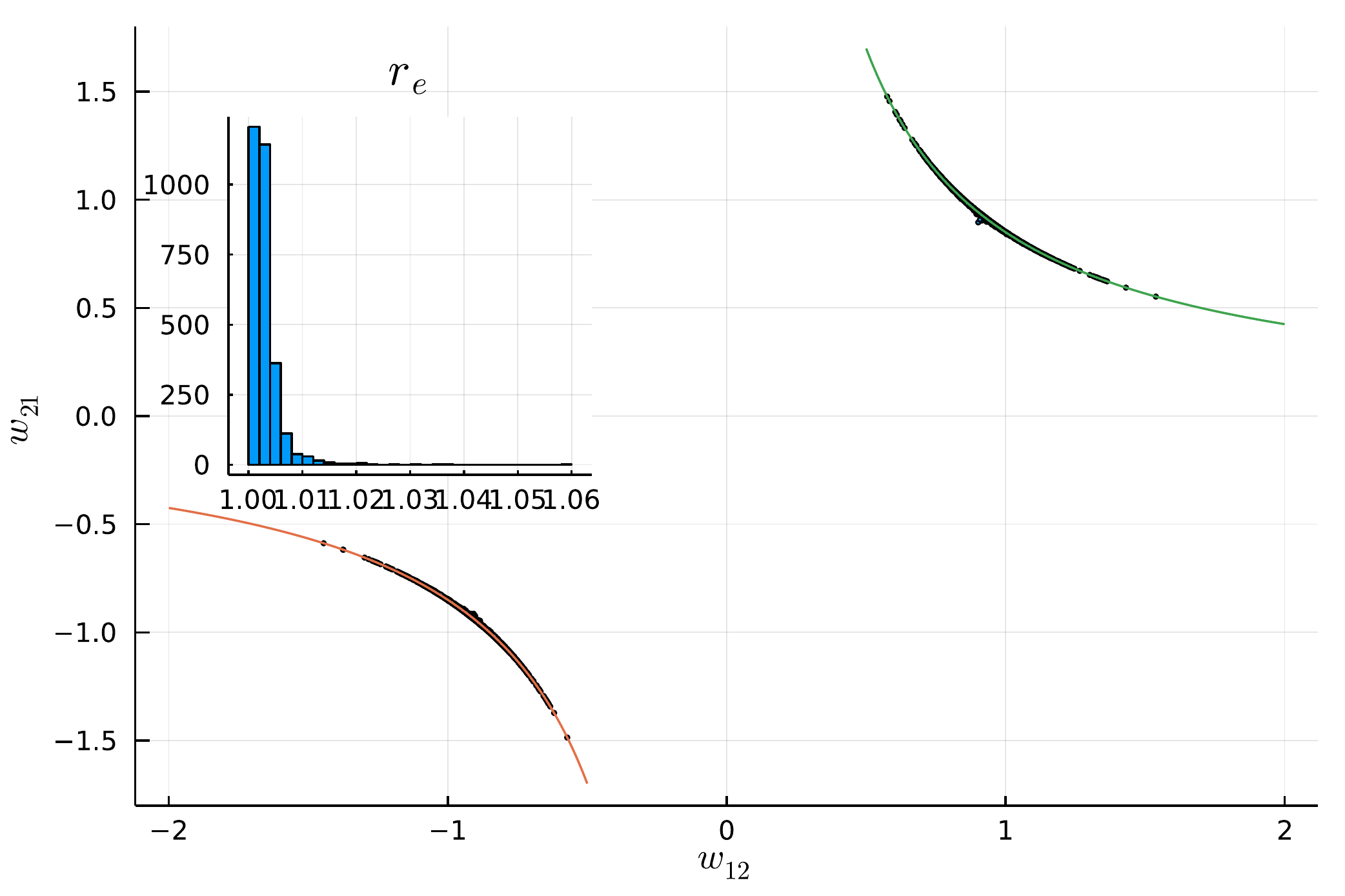}  
  \caption{$N=\infty$}
  
\end{subfigure}
\caption{Outcomes of batch simulations under $N=5,10,20$ and $\infty$, respectively. The black dots indicate training outcomes, while the red and green hyperbola lobes visualize rank one minimizers. Sample distributions of effective rank are also included in the embedded graphs.}
\label{fig:diagonalmatrixcompletion5Inf}
\end{figure}

\begin{figure}[h!]
\begin{subfigure}{.48\textwidth}
  \centering
  \includegraphics[width=1.0\linewidth]{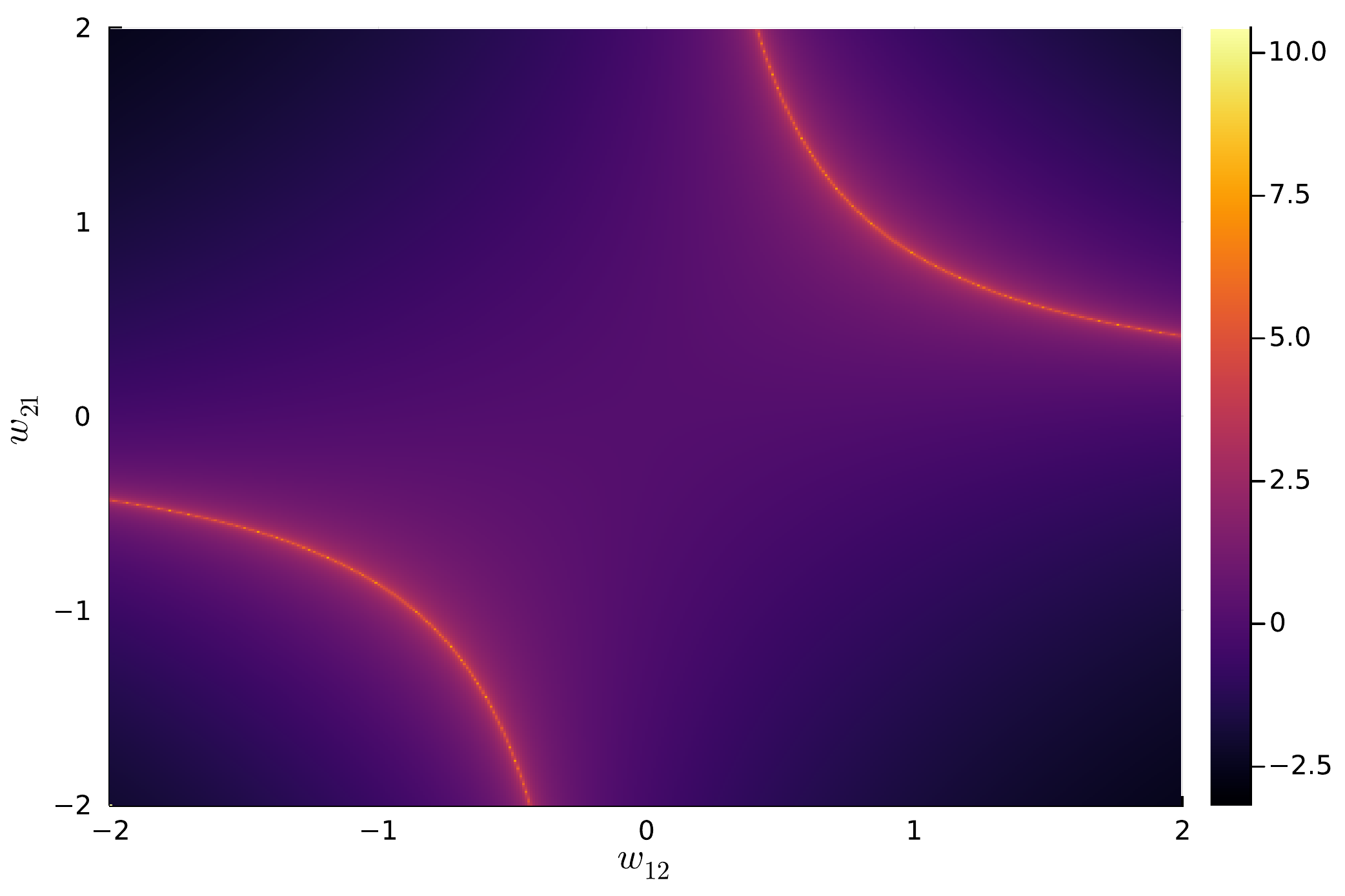}  
  \caption{}
  
\end{subfigure}
\hspace{0.02\textwidth}
\begin{subfigure}{.48\textwidth}
  \centering
  \includegraphics[width=1.0\linewidth]{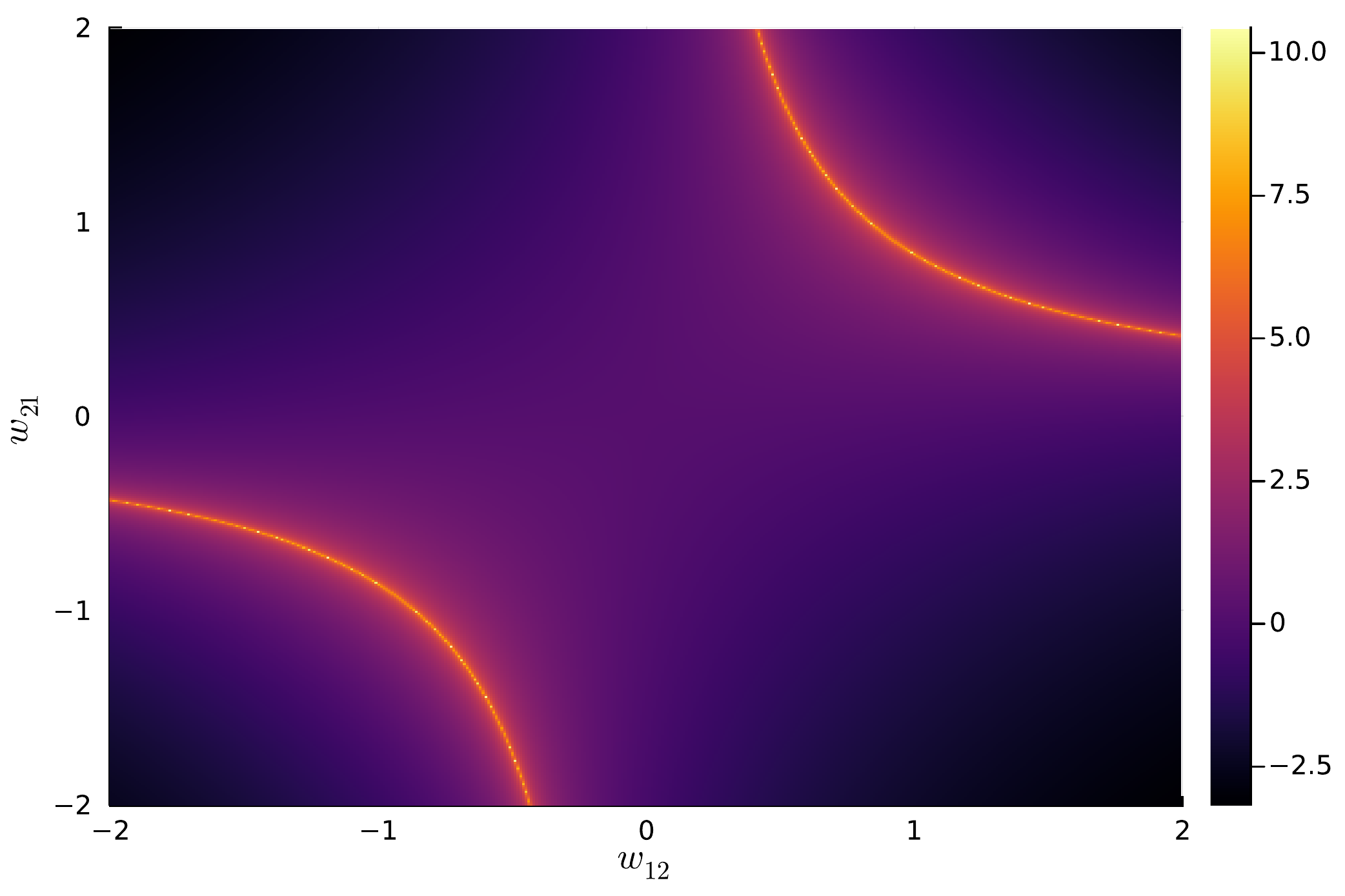}
  \caption{}
  
\end{subfigure}
\caption{Heatmap of the logarithmic volume on the plane of global minimizers, for $N=5$ and $N=\infty$, respectively. Note that the volume itself does not depend on the loss function, only the plane where the section was taken.}
\label{fig:logvolumeforms}
\end{figure}

\begin{figure}[h!]
    \centering
    \includegraphics[scale=0.6]{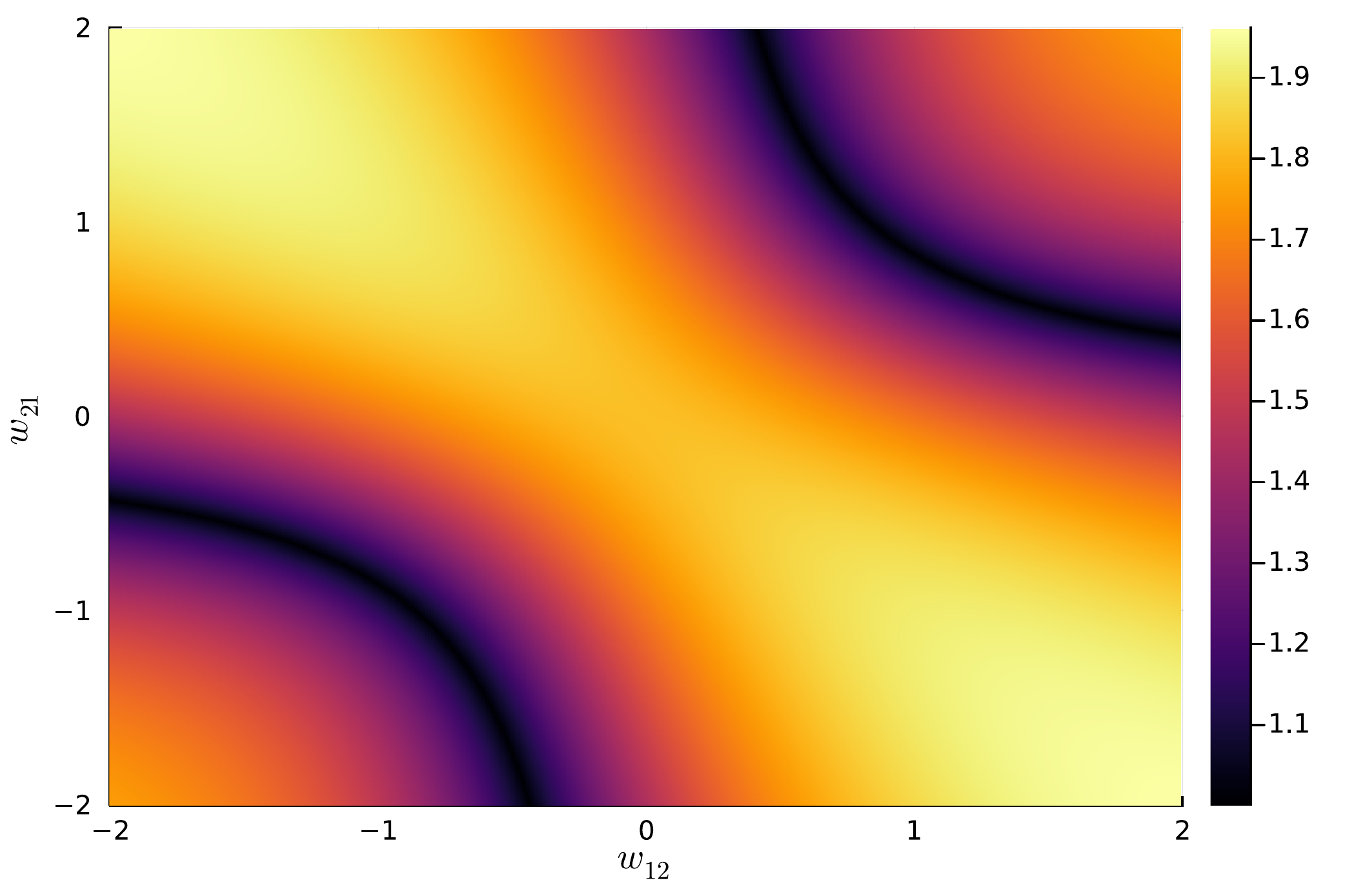}
    \caption{Effective rank of global minimizers.}
    \label{fig: effectiverank}
\end{figure}

Based on the example in section \ref{sec:wideexample} and the results in Figure \ref{fig:diagonalmatrixcompletion5Inf} we conclude that the infinite depth limit shows very similar behavior qualitatively to finite, sufficiently large depth models. Thus it is worth studying and a predictive theory of implicit regularization must be consistent with the infinite depth model.

Comparing Figure \ref{fig:diagonalmatrixcompletion5Inf} and Figure \ref{fig: effectiverank}, we see that there is no clustering of simulation outcomes in high effective rank regions. However, effective rank does not take the effect of depth into consideration. As seen in Figure \ref{fig:logvolumeforms}, state space volume shows higher concentration at higher depth, just as simulation outcomes in higher depth show more concentration in these regions. While we are not ready to make this relation quantitative, these results suggest that state space volume is a good candidate to rely on for a quantitative theory of implicit regularization.

We also see that the accumulation of training outcomes is strong near the corners of the hyperbola of rank one completions. If high volume is a predictive quantity for training outcomes, we expect a subtle decrease of volume along the hyperbola starting from the corners. We characterize the blowup rates along the hyperbola by computing the normal perturbation of singular values and volume density along the hyperbola. 

To simplify calculations, we fix the diagonal elements to 1, getting the one-parameter family of rank-one completions
    \begin{equation}
        W=\begin{bmatrix}
        1 & \gamma \\ \frac{1}{\gamma} & 1
        \end{bmatrix}.
    \end{equation}

    To simplify some of the formulas, we assume $\gamma>0$, restricting the analysis on the positive lobe of the hyperbola. The singular value decomposition $W(\gamma)=U(\gamma)\Sigma(\gamma) V^T(\gamma)$ can be explicitly  computed:
    \begin{equation}
        U=\frac{1}{\sqrt{1+\gamma^2}}\begin{bmatrix}
            \gamma & -1 \\ 1 & \gamma
        \end{bmatrix},\,\,
        \Sigma=\begin{bmatrix}
            \gamma+\gamma^{-1} & 0 \\0 &0
        \end{bmatrix},\,\,
        V=\frac{1}{\sqrt{1+\gamma^2}}\begin{bmatrix}
            1 & \gamma \\ \gamma & -1
        \end{bmatrix}.
    \end{equation}
    
    We will consider perturbations of size $\eta$ in the normal direction of the hyperbola $1/\gamma$,
    \begin{equation}
        \dot W= \frac{\eta}{\sqrt{1+\gamma^{-4}}} \begin{bmatrix}
            0 & \gamma^{-2} \\ 1 & 0
        \end{bmatrix}.
    \end{equation}

    Using Lemma \ref{smooth_svd}, the perturbed singular values can be computed,
    \begin{equation}
        \Sigma \left(\gamma, \eta\right)= \begin{bmatrix}
            \gamma+\gamma^{-1}+\frac{2\eta}{\left(1+\gamma^2\right) \sqrt{1+\gamma^{-4}}} & 0 \\ 0 & \eta\frac{ \sqrt{\gamma^4+1}}{\gamma^2+1}
        \end{bmatrix}+ \mathcal{O} \left(\eta^2 \right).
    \end{equation}

    And consequently, up to leading order, the determinant is
    \begin{equation}
        \det \left(\Sigma\left(\gamma,\eta\right)\right) = \eta \frac{\sqrt{\gamma^4+1}}{\gamma}+ \mathcal{O}(\eta^2).
    \end{equation}

    Plugging this back into the volume density function from equation \eqref{eq:volumedW},
    \begin{equation}
    \begin{split}
        \frac{\mathrm{van}\left(\log \Sigma^2\right)}{\det \left(\Sigma\right) \mathrm{van} \left(\Sigma^2\right)}&=\frac{2\gamma \left(\log \left( \gamma +\gamma^{-1}+\frac{2\eta}{\left(\gamma^2+1\right)\sqrt{1+\gamma^{-4}}}\right)-\log \left(\eta \frac{\sqrt{\gamma^4+1}}{\gamma^2+1}\right)\right)}{\left(\eta \sqrt{\gamma^4+1}+\mathcal{O}\left(\eta^2\right)\right) \left(\left(\gamma+\gamma^{-1}+\frac{2\eta}{\left(1+\gamma^2\right) \sqrt{1+\gamma^{-4}}}\right)^2-\left(\eta\frac{ \sqrt{\gamma^4+1}}{\gamma^2+1}\right)^2\right)}\\&= \mathcal{O} \left(\frac{|\log \eta|}{\eta}\right)
    \end{split}    
    \label{eq:volumescaling2by2}
    \end{equation}
    as $\eta \rightarrow 0$. We gained quantitative understanding of how the singular values perturb. Specifically, we learn that the leading order coefficient of the smaller singular value, $\sigma_2$, is $\frac{ \sqrt{\gamma^4+1}}{\gamma^2+1}$. This function has a minimum at $\gamma=1$, the corner of the hyperbola. This shows s quantitative, but not qualitative decrease of volume along the hyperbolas. The quickest blowup of the volume density is at the corners of the hyperbola.

\subsection{Other configurations}

\subsubsection{Example: Single rank deficient minimizer}
\label{sec:finiterankdefminimizers}

In the previous examples, both global minimizers and rank deficient global minimizers were nonunique. Here we provide an example, where minimizers are nonunique, but there is only one rank deficient (and high state space volume) minimizer. The setup is the following: the energy function is $E_{\mathcal{T}}(W)$, 
\begin{equation}
    \mathcal{T}=\begin{bmatrix}
        1 & 1 \\ 0 & 1
    \end{bmatrix},
\end{equation}
and $N=\infty$. In this case, the minimizers are of the form
\begin{equation}
    \begin{bmatrix}
        \Phi_{11} & \Phi_{12} \\ w_{21} & \Phi_{11}
    \end{bmatrix},
\end{equation}
for arbitrary $w_{21}$, while imposing rank deficiency gives one solution, $w_{21}=\Phi_{11} \Phi_{22}/ \Phi_{12}$. All simulation outcomes of 1000 random initial conditions drawn from $\text{Wigner}(0,0.001)$ showed convergence to this minimizer. As illustration, 5 sample trajectories are shown in Figure \ref{fig:uniqerankdeficientminimizer}.

\begin{figure}[h!]
    \centering
    \includegraphics[scale=0.6]{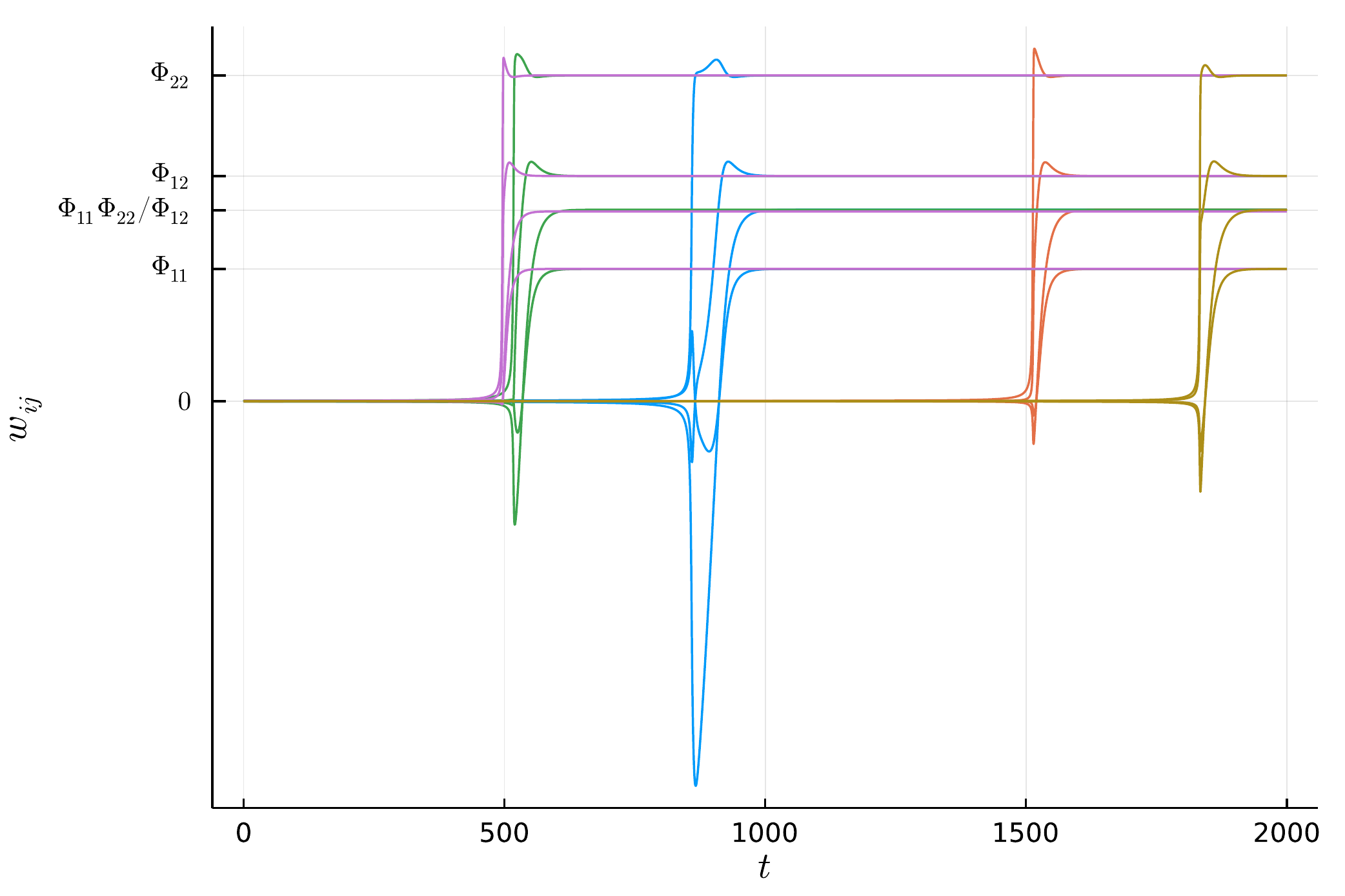}
    \caption{Five sample trajectories of the matrix elements under training are shown. All of them converging to the rank deficient minimizer.}
    \label{fig:uniqerankdeficientminimizer}
\end{figure}

\subsubsection{A 3-by-3 example}
\label{sec:3by3example}

We continue with a more complicated, $N=\infty$, $d=3$ example. Computing the minimal rank to which a partially known matrix can be completed is in general nontrivial. For a detailed discussion of the computational complexity of matrix completion, see \cite{Hardt2014ComputationalLF}. We covered cases before, where either the small matrix size, or the symmetry of the configuration in the observed elements (diagonal matrix completion) makes this problem trivial.

The equations for the missing elements can be written out by (symbolic) LU factorization under the assumption of a given rank. If the equations have a solution for the unknown elements, the matrix has a completion to the given rank. Solving the resulting system of multivariate polynomials may not be possible.

We show this process for an easily solvable case, take for example the first step of LU factorization for the following configuration of $\Phi_{ij} \neq 0$ and arbitrary $w_{ij}$ we are trying to solve for:

\begin{equation}
\begin{split}
    M(w_{31},w_{12},w_{23})&=\begin{bmatrix}
        \Phi_{11}& w_{12}& \Phi_{13} \\
        \Phi_{21}& \Phi_{22}& w_{23}\\
        w_{31}& \Phi_{32}& \Phi_{33} 
    \end{bmatrix}
    \\&=
    \begin{bmatrix}
        \Phi_{11} \\
        \Phi_{21}\\
        w_{31} 
    \end{bmatrix}
    \begin{bmatrix}
        1& \frac{w_{12}}{\Phi_{11}}& \frac{\Phi_{13}}{\Phi_{11}} \\
    \end{bmatrix}
    +
    \begin{bmatrix}
        0& 0& 0 \\
        0& \Phi_{22}-\frac{\Phi_{21} w_{12}}{\Phi_{11}}& w_{23}-\frac{\Phi_{13}\Phi_{21}}{\Phi_{11}}\\
        0& \Phi_{32}-\frac{w_{31} w_{12}}{\Phi_{11}}& \Phi_{33} - \frac{w_{31} \Phi_{13}}{\Phi_{11}}
    \end{bmatrix}
\end{split}
    \label{eq:discreterankdeficient}
\end{equation}
We see that there is no solution for the rank one completion and that the rank-two completions are implicitly given by 
\begin{equation}
0=\Phi_{33}- \frac{\Phi_{13}}{\Phi_{11}} w_{31} - \frac{\left(\Phi_{32}-\frac{w_{31}w_{12}}{\Phi_{11}}\right) \left(w_{23}-\frac{\Phi_{13}\Phi_{21}}{\Phi_{11}}\right)}{\Phi_{22}-\frac{\Phi_{21} w_{12}}{\Phi_{11}}}.
\label{eq:ranktwocompletions}
\end{equation}

In particular, an easily identifiable rank-two completion is given by the choice
\begin{equation}
    w^1_{12}=\frac{\Phi_{32} \Phi_{13}}{\Phi_{33}},\, w^1_{23}=\frac{\Phi_{13}\Phi_{21}}{\Phi_{11}},\, w^1_{31} =\frac{\Phi_{33}\Phi_{11}}{\Phi_{13}}.
    \label{eq:thepoint}
\end{equation}
Let $M$ denotee the zero energy matrix of with coordinates \eqref{eq:thepoint}.

In order to run numerical simulations, we pick arbitrary matrix elements for the target matrix:
\begin{equation}
\Phi=
    \begin{bmatrix}
        -1.55795& \cdot& 1.58397 \\0.212869& 0.0337805& \cdot\\\cdot &1.32488 &1.92653
    \end{bmatrix}.
\end{equation}

\begin{figure}[h!]
    \centering
    \includegraphics[width=0.45\textwidth]{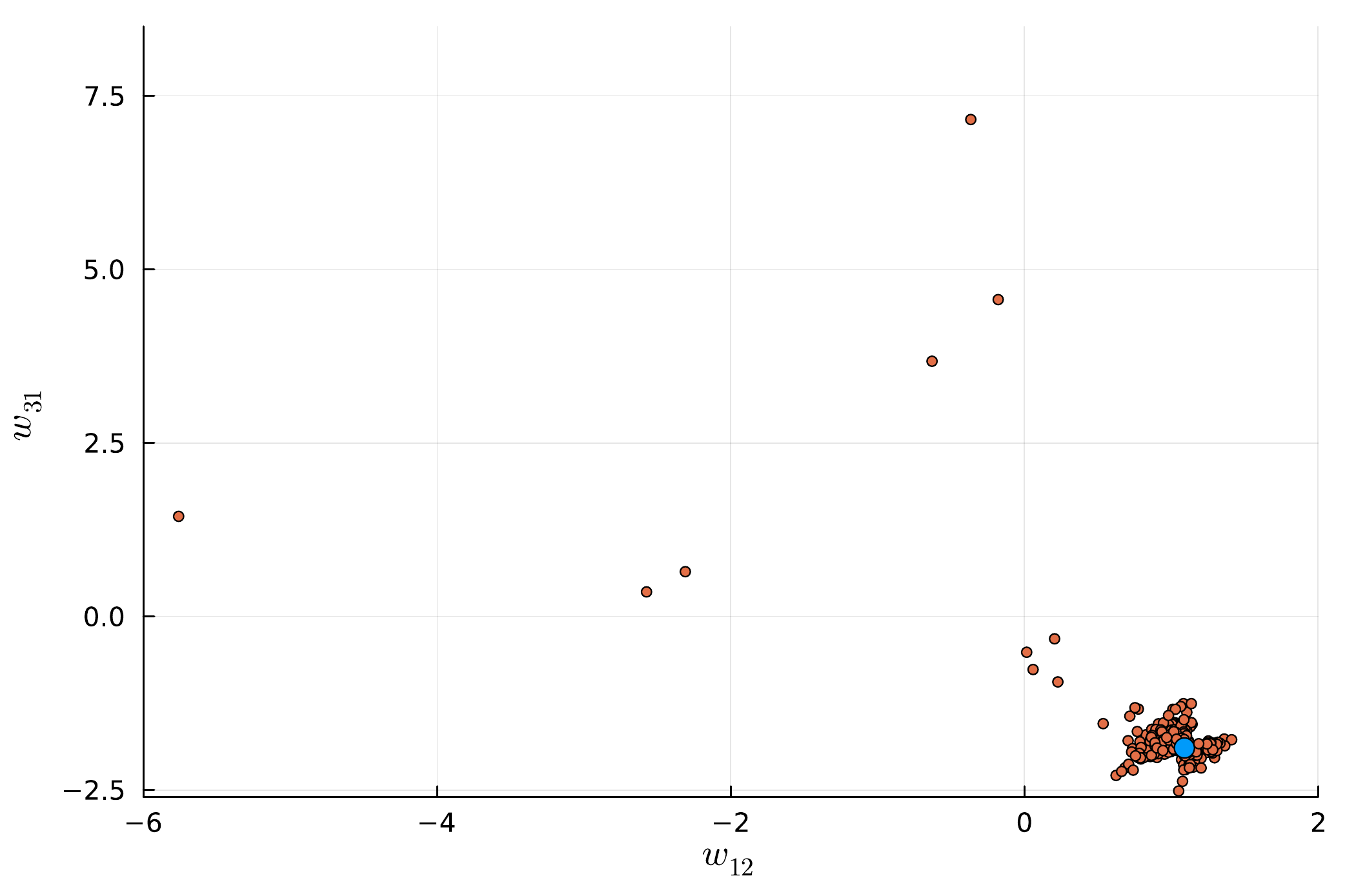}
    \hspace{0.05\textwidth}
    \includegraphics[width=0.45\textwidth]{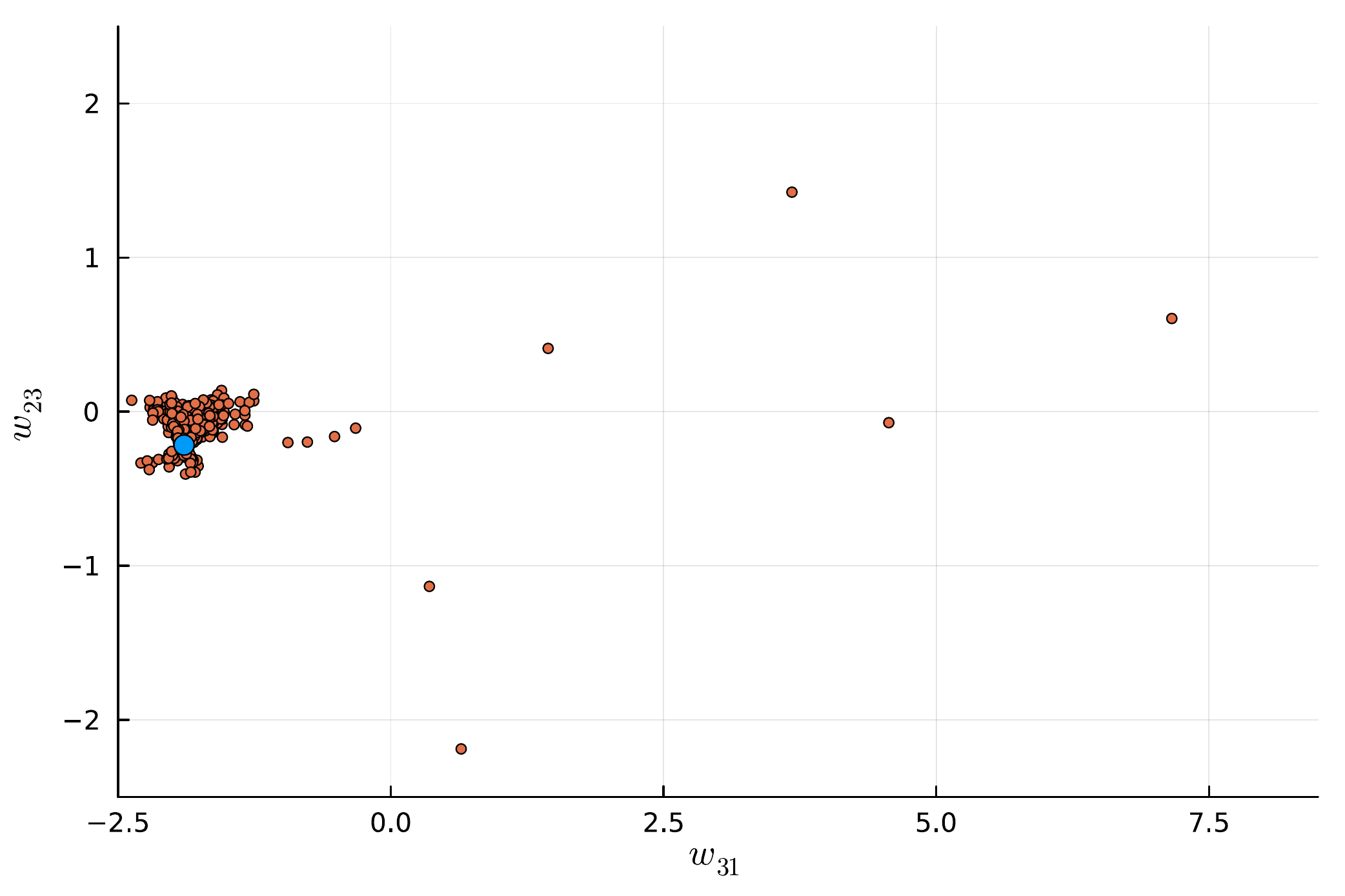}
    \caption{Clustering of training outcomes obtained in the setup of equation \eqref{eq:discreterankdeficient}. The majority of the 500 outputs cluster near one particular rank-two minimizer indicated by the blue dot. Random initial conditions were drawn from $\text{Wigner}(0,0.02)$.}
    \label{fig:3by3example}
\end{figure}

Figure \ref{fig:3by3example} shows the outcome of strong clustering in a region near $M$. We interpret this as a shortcoming of the rank based predictions, as the entire two parameter family of minimizers defined by \eqref{eq:ranktwocompletions} has rank-two. We attempt to explain the observation using phase space volume. We approximate phase space volume within the zero energy plane spanned by coordinates $(w_{12},w_{31}, w_{23})$ upon simple Monte Carlo integration. We integrate the volume form \eqref{eq:volumeforminf} in a small cube around some rank-two minimizers. 

Since these points are singular, volume form \eqref{eq:volumeforminf} blows up. As a result of this, the volume in a domain containing these points can be infinity. The volume, however, is locally finite, so working with domains bounded away from singular points allows us to make quantitative comparison between regions of the state space. Given a rank-two minimizer, we can study the concentration of volume around it by simple monte carlo integration of the volume density \eqref{eq:volumedW}.

Define the set of three-by-three matrices with smallest singular value strictly larger than $h$,
\begin{equation}
    GL_3^h=\{X \in GL(3):\sigma_3\left(X\right)>h\},
\end{equation}
and the three dimensional $H$-cube around a point $X$
\begin{equation}
    \mathcal{C}_H(X)=\{W \in \mathbb{R}^{d\times d}:\|x_{31}-w_{31}\|_1<H/2,\|x_{12}-w_{12}\|_1<H/2,\|x_{23}-w_{23}\|_1<H/2\}.
\end{equation}

We can generate uniform random points in $GL_3^h \cap \mathcal{C}_H(X)$ by drawing from a uniform distribution on $\mathcal{C}_H(X)$ and discarding the point when its smallest singular value is smaller than $h$. 

To carry out these integrations, random rank-two minimizers with undetermined elements $(w_{12},w_{31}, w_{23})$ solving \eqref{eq:ranktwocompletions} are also required. We obtain these matrices by sampling $w_{12}$ and $w_{31}$ from a centered normal distribution of standard deviation 10. Parameters $H=0.001$ and $h=0.00001$ were chosen.

\begin{figure}[h!]
    \centering
    \includegraphics[width=0.65\textwidth]{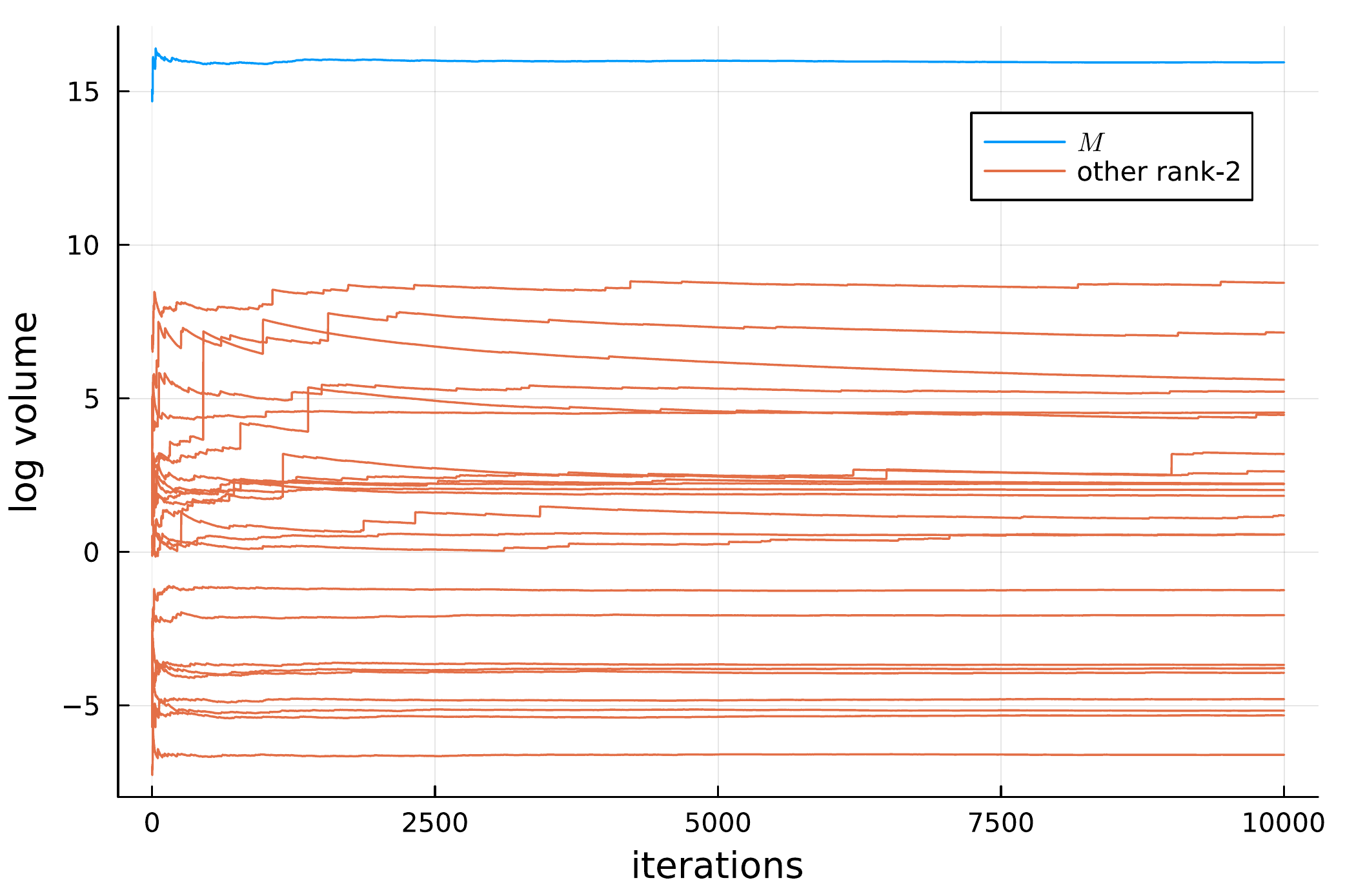}
    \caption{Monte Carlo integration of volume around rank-two minimizers.}
    \label{fig:montecarlo}
\end{figure}

The outcome of simulations conducted on $M$ and 24 other randomly selected rank-two minimizers is shown in Figure \ref{fig:montecarlo}. $M$ is the point of the highest volume vicinity, the difference to other equal rank minimizers spanning several orders of magnitude. This is in alignment with the clustering of training outcomes. 

This example motivates the rigorous analysis of the asymptotics of the volume form around singular matrices. While we observe large quantitative differences of volume in the cluster of training outcomes, notice that $M$ was a point arbitrarily selected in this cluster. Based on these results, we cannot conclude that there are no points with even larger volume around them, or even points with higher blowup rates. Our results nonetheless strongly suggest that state space volume is the right tool to make predictions on training outcomes, when rank based predictions are inconclusive.

\subsubsection{Example: Minimal Rank State Space}
\label{sec:rankoneexample}

Our last example is a pathological case showing an explicit shortcoming of the rank-hypothesis of implicit regularization. Let $\mathcal{R}^1=\{X \in \mathbb{M}_{2}: \text{rank}\left(X\right)=1\}$. It is shown in Chapter 1 in \cite{Bah2019} that the DLN geometry for $N<\infty$ can be defined on the manifold of fixed rank matrices using equation \eqref{metric} for the metric.

It is easy to show that the volume form of this metric is 
\begin{equation}
    \sqrt{G^N_1(W)} = \frac{N}{\sigma^{3\frac{N-1}{N}}},
    \label{jokevolumeform}
\end{equation}
where $\sigma$ denotes the single nonvanishing singular value of $W \in \mathcal{R}^1$. This volume form does not blow up at any point other than points near the origin. 

For numerical simulations, take the setup of section \ref{hyperbolas}, save for setting $N=20$ and the change in the generation of random initial conditions. We generate samples $u \sigma v^T=W'_0\sim \text{Wigner}(0,0.001)$, then use
\begin{equation}
    W_0=u \begin{bmatrix}
        \sigma_{1} &0 \\ 0& 0
    \end{bmatrix}
    v^T
\end{equation}
as initial conditions. The set of minimizers (same as minimum rank minimizers) consists of matrices 

    \begin{equation}
    \begin{bmatrix}
        \Phi_1 & w_{12} \\ \frac{\Phi_{2} \Phi_{1}}{w_{12}} & \Phi_2
    \end{bmatrix}
    =
    \begin{bmatrix}
        0.58724 & w_{12} \\ \frac{0.8497}{w_{12}} & 1.447
    \end{bmatrix}.
\end{equation}
Volume form \eqref{jokevolumeform} shows that highest volume is obtained at smallest $\sigma$, but notice that in this case the volume form does not blow up. Nonetheless, we can predict highest concentration of training outcomes at the minimal singular value completions. Simple computation verifies that these are the same as the symmetric completions,
\begin{equation}
    \begin{bmatrix}
        0.58724 & \pm0.921811 \\ \pm0.921811 & 1.447
    \end{bmatrix},
\end{equation}
which both admit $\sigma_{\min}=2.03424$. Thus, based on phase space volume, matrices  with singular value at $\sigma_{\min}$ and slightly above are expected to have high representation among training outcomes. Figure \ref{fig:rank1distro} shows the outcome of 1072 simulations using random initial conditions verifying the predictions. 

Notice furthermore that $r_e(W)=1$ for all $W\in \mathcal{R}^1$, therefore it is impossible to make any predictions based on effective rank.

\begin{figure}[H]
    \centering
    \includegraphics[scale=0.6]{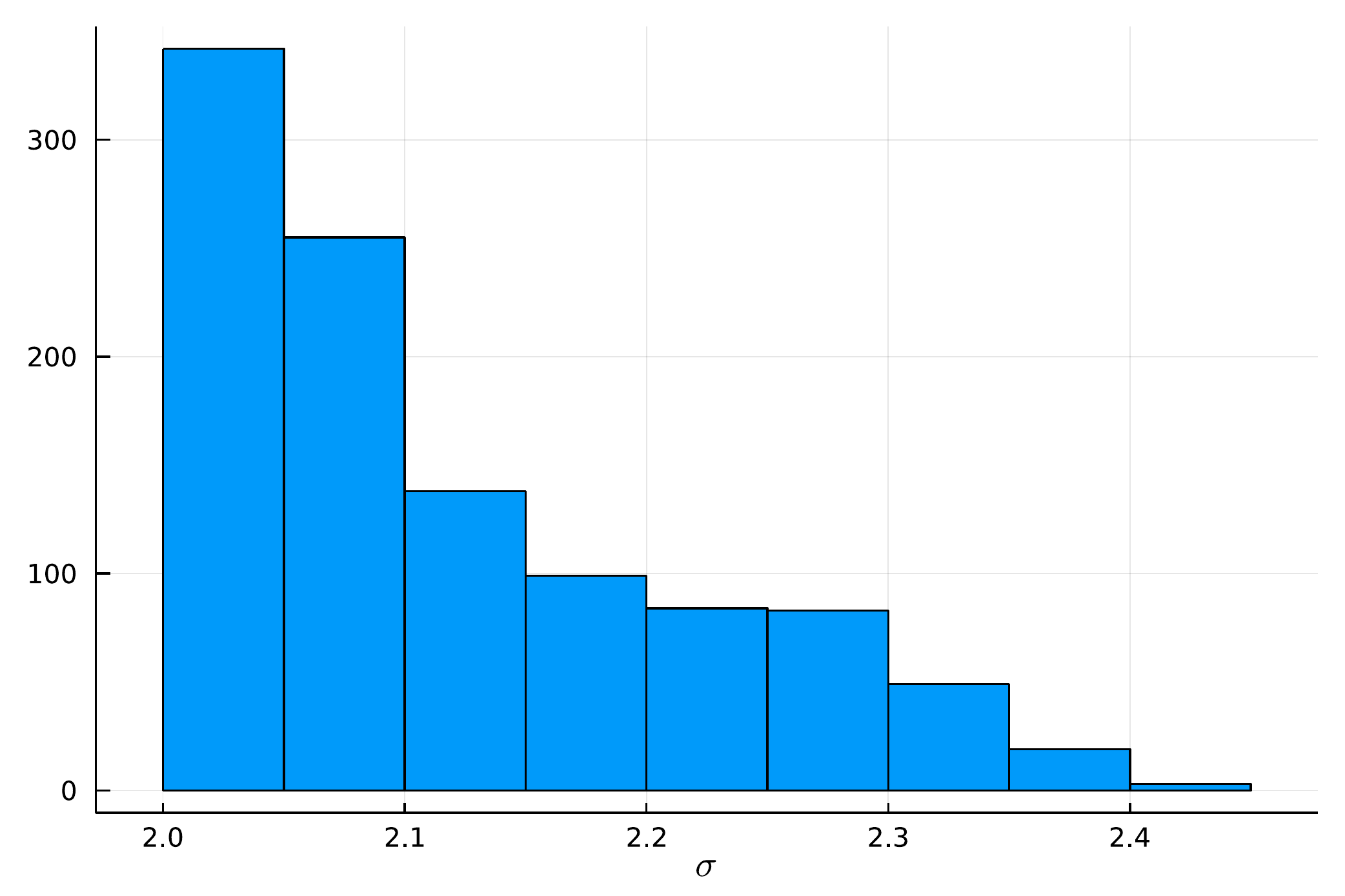}
    \caption{Sample distribution of 1072 simulation outcomes of rank one, diagonal matrix completion.}
    \label{fig:rank1distro}
\end{figure}

\section{Conclusion}
\label{sec:conclusion}
We have used Riemannian geometry to improve the rank-based understanding of implicit regularization in DLN. We derived new representations for the DLN metric and formalized the infinite depth limit. The most salient feature of our analysis is the derivation for the volume forms under depths of all positive integers, as well as infinity.

Using our results on the geometry of the DLN, we also improved upon the understanding of training dynamics. We found formulas for linear attraction rates using the eigendecomposition of the DLN metric. We then proved local normal hyperbolicity for the critical manifold of training dynamics under a relatively general family of loss functions.

Looking forward, we would like to continue our study on the DLN geometry. The simplicity of the formulas presented in this paper suggest that further intrinsic quantities could be expressed using explicit formulas. The derivation of higher order geometric quantities could give rise to new results on the dynamics. The precise formulation of stochastic training dynamics requires the notion of Brownian motion on a Riemannian manifold. The stochastic differential equation describing intrinsic Brownian motion relies on computations of the Levi-Civita connection and curvatures \cite[Chapter~3]{Hsu2002}.

We engineered low dimensional examples to verify our idea, that implicit regularization is explained by high state space volume. We showed that training in the case of $N=\infty$ shows implicit regularization the same way as has been established for $N<\infty$. To summarize, the DLN at $N=\infty$ provides a novel model of implicit regularization which is simpler than the $N<\infty$ counterpart.

\section*{Acknowledgments}
The work of GM and ZV was supported in part by NSF grant  DMS-2107205. The work of NC was supported in part by the Israel Science Foundation (grant 1780/21) and the Tel Aviv University Center for AI and Data Science (TAD).

\bibliographystyle{siamplain}
\bibliography{references}

\end{document}